\documentclass{article}
\usepackage{amssymb,amsmath, amsthm}
\usepackage{enumerate}
\usepackage{graphicx}
\usepackage{geometry}
\newtheorem{theorem}{Theorem}

\newtheorem{lemma}{Lemma}
\newtheorem{example}{Example}
\newtheorem{observation}{Observation}
\newtheorem{conjecture}{Conjecture}
\newcommand{\N}{\ensuremath{\mathbb N}}
\newcommand{\SG}{\ensuremath{\mathcal{SG}}}
\newcommand{\opt}{\mbox{opt}}

\newcommand{\bs}{\ensuremath \mathrm}

\begin{document}
\title{Game values of arithmetic functions}
\author{Douglas E. Iannucci\footnote{University of the Virgin Islands.}\; and Urban Larsson\footnote{National University of Singapore, urban031@gmail.com.}}
\maketitle
\begin{abstract}
Arithmetic functions in Number Theory meet the Sprague-Grundy function from Combinatorial Game Theory. We study a variety of 2-player games induced by standard arithmetic functions, such as Euclidian division, divisors, remainders and relatively prime numbers, and their negations. \end{abstract}
%%
%%%    SECTION 1.     Introduction
%%
%
\section*{Introduction}
Consider the following situation: two players, Alice and Bob, alternate to partition a given finite number of positive integers into components of the form of a non-trivial Euclidian division. Whoever will fail to follow the rule, because each number is a ``1'', loses the game. You are only allowed to split one number at a time. For example, if Alice starts from the number $7$, then her options are $1+\cdots +1$, $2+2+2+1$, $3+3+1$, $4+3$, $5+2$ and $6+1$; here the `+' sign from arithmetic functions becomes the disjunctive sum operator (a convenient game component separator) in the game setting. By observing that we may remove any pair of the same numbers (by mimicking strategy), and we may remove a one unless the option is the terminal position (since its set of options is empty), the set of options from $7$ simplifies to $\{1,2, 4+3, 5+2,6\}$. Suppose now that Alice starts playing from the disjunctive sum $7+2$. By the above analysis its easy to find a winning move to $2+2+2+1+2$. What if she instead starts from the composite game $7+3$?

We study 2-player normal-play games defined with the nonnegative or positive integers as the set of positions. The two players alternate turns and if a player has no move option then he/she loses. At each stage of game, the move options are the same independent of who is to play. In combinatorial game theory, this notion is referred to as {\em impartial}. Games terminate in a finite number of moves, and there is a finite number of options from each game position; i.e., games are {\em short}. This allows us to use the famous theory discovered independently by Sprague~\cite{Sprague} and Grundy~\cite{Grundy}, which generalizes normal-play {\sc nim}, analyzed by Bouton~\cite{Bouton}, into disjunctive sum play of any finite set of impartial normal-play games.

Arithmetic functions are at the core of number theory, in a similar sense that {\sc nim} and Sprague-Grundy are central to the theory of combinatorial games. We consider {\em arithmetic functions} \cite{HW} of the form $f: X\rightarrow Y$, where the set $X$ is either the nonnegative integers, $\mathbb{N}_0=\{0,1,\ldots\}$, or the positive integers, i.e., the natural numbers $\mathbb{N}=\{1,2,\ldots\}$, and where typically $Y=2^X$ is the set of all subsets of the nonnegative or positive integers respectively. In some settings, for example when the arithmetic function is counting the instances of another arithmetic function, we may take $Y =\mathbb{N}_0$; in these cases, we will refer to $f$ as a counting function, and the games as {\em counting games}. %  

Arithmetic functions may conveniently be interpreted as {\em rulesets} of impartial games, and here we present old and novel games within a classification scheme. 
Each arithmetic function induces a couple of rulesets, and we let $\mbox{opt}:X\rightarrow Z$, define the set of move options from $n\in X$, given some arithmetic function $f$, with sometimes an imposed  terminal sink, or a modified codomain, for example, in the {\em powerset games} to come, $Z=2^{2^X}$; in the singleton and counting game cases we may take $Z=Y$ (modulo sometimes adjustment for $0$).

More specifically, we interprete the arithmetic functions in terms of heap games; for each game position, i.e., heap size represented by a number (of pebbles). There are two main variations.
\begin{itemize}
    \item A player can \emph{move-to} a number or a disjunctve sum of numbers, induced by the arithmetic function. 
    \item A player can \emph{subtract} a number or a disjunctive sum of numbers, induced by the arithmetic function.
\end{itemize}
We will use the name of the arithmetic function and prepend the letters 
{\sc M} or {\sc S}, respectively, for the move-to and the subtract versions of a particular game. 
The following two examples are excerpts from Section~\ref{sec:aliquots}.
\begin{example}
If the players may move-to a divisor, we get for example: from $\bs{6}$ the options are $1$,$2$ and $3$. From $7$ you may only move to $1$. Here, the divisors must be proper divisors. %Thus, when playing the disjunctive sum $6+7$, the winning moves are to $2+7$, or $3+7$.
\end{example}

\begin{example}
If the players may subtract a divisor, we get for example: 
from $6$, the options are $5, 4, 3$ and $0$. From $7$ you can move to $0$ or $6$. %Thus, a winning move in the disjunctive sum $6+7$ is to $6+6$. Other winning moves are to $3+7$ or $5+7$.
\end{example}

Before this paper, instances where number theory connects with impartial games might individually have seemed like `lucky cases'. However, we feel that the relatively large number of such examples justifies a more systematic study.\footnote{See classical works such as Winning Ways \cite{WW} for results on impartial games coinciding with number theory.} 

Let us list some game rules induced by some standard arithmetic functions. When there is only one single option, we may omit the set brackets.

\begin{enumerate} 
\item The aliquot (divisor) games: 
	\begin{enumerate}
	\item {\sc maliquot}. 
	Move-to a proper divisor of your number; i.e.,
	$$\mbox{opt}(n)=\{d:d\mid n,\; 0<d<n\}.$$
	\item {\sc saliquot}.
	Subtract a divisor of your number; i.e.,
	$$\mbox{opt}(n)=\{n-d:d\mid n,\; d>0\}.$$
	\end{enumerate}
 
\item The aliquant (nondivisor) games: 
	\begin{enumerate}
	\item {\sc maliquant}.
	Move-to a nondivisor of your number; i.e.,
	$$\mbox{opt}(n)=\{k:1\le k\le n,\;k\nmid n\}.$$
	\item {\sc saliquant}.
	Subtract a nondivisor of your number; i.e.,
	$$\mbox{opt}(n)=\{n-k:1\le k\le n,\;k\nmid n\}.$$
	\end{enumerate}

\item The $\tau$-games:\footnote{Here  $\tau(n)$ counts the natural divisors of~$n$: $\tau(n) = \sum_{d\mid n}1$.  The divisor function $\tau$ is multiplicative, with $\tau(p^a)=a+1$ for all primes~$p$ and natural numbers~$a$. In one of our game settings, we use instead the number of proper divisors, and so let $\tau'=\tau-1$, so that in particular $\tau'(1)=0$ and $\tau'(2)=1$ (here we lose multiplicativity).}
	\begin{enumerate}
	\item {\sc mtau}.
	Move-to the number of proper  divisors of your number; i.e.,
	$$\mbox{opt}(n)=\tau'(n).$$
	\item {\sc stau}.
	Subtract the number of divisors of your number; i.e.,
	$$\mbox{opt}(n)=n-\tau(n).$$
	\end{enumerate}

\item The totative (relative prime residue) and the nontotative games:\footnote{The `move-to' and `subtract' variations are the same, because $(k,n)=1$ if and only if $(n-k,n)=1$.}
	\begin{enumerate}
	\item {\sc totative}.
	Move-to any relatively prime residue; i.e.,
	$$\mbox{opt}(n)=\{k:1\le k\le n:(k,n)=1\}.$$
	\item {\sc nontotative}.
	Move-to any smaller residue that is not relatively prime to your number; i.e.,
	$$\mbox{opt}(n)=\{k:1\le k< n:(k,n)>1\}.$$
	\end{enumerate}

\item The totient ($\phi$) games: 
\begin{enumerate}
	\item {\sc totient}.
	Move-to the number of relatively prime residues modulo your number; i.e.,
	$$\mbox{opt}(n)=\phi(n).$$
	\item {\sc nontotient}.
	Instead subtract this number; i.e.,
	$$\mbox{opt}(n)=n-\phi(n).$$
	\end{enumerate}

\item {\sc dividing}. Divide your number into a maximum number of equal parts, at least two; i.e.,
$$\mbox{opt}(n)=\{\; \underbrace {k+k+\cdots + k}_{m\;\text{$k$'s}} : km= n,\; m>1\}.$$

\item {\sc dividing-and-remainder}. Divide your number into a number of equal parts and a remainder, which is smaller than the other parts and possibly~0; i.e.,
$$\mbox{opt}(n)=\{\, \underbrace {k+k+\cdots +k+r}_{m\;\text{$k$'s}}\} : km+r=n,\; m>0,\;0\le r<k\,\}.$$
This game has two simpler variations, as defined in Section~\ref{sec:divrem}. 

\item {\sc factoring}. Factor your number into at least two components, and at most the number of prime factors, counting multiplicity; i.e.,
$$\mbox{opt}(n)=\{a_1 + a_2 + \cdots + a_k:1<a_1\le a_2\le\cdots\le a_k,\,a_1a_2\cdots a_k=n\,\}.$$

\end{enumerate}

Item 7 here is the game in the first paragraph of the paper. The goal of this paper is to evaluate the {\em nim-values}, a.k.a. Sprague-Grundy values, of these games, and  hence winning strategies can be computed, via the nim-sum operator, in disjunctive sum with any other normal-play game. The nim-values of a ruleset are defined via the \emph{minimal excludant function}, $\mbox{mex}:2^X\rightarrow \N_0$, where $X$ is the set of nonnegative (or positive)  integers. Let $A\subset \N_0$ be a strict subset of the nonnegative integers. Then $\mbox{mex}(A)=\min\{x : x\in \N_0\setminus A\}$ and $\SG(n)=\mbox{mex}\{\SG(x):x\in \mbox{opt}(n)\}$. Note that, if there is no move option from $n$, then $\SG(n)=0$.  Recall that the {\em nim-sum} is used to compute the nim-value of a disjunctive sum of games, i.e., $\SG(\sum n_i)=\bigoplus\SG(n_i)$, where `$\sum$' is disjunctive sum operator, and `$\bigoplus$' is the sum modulo $2$ without carry of the numbers $n_i$ in their binary representations. 

If $f$ is a counting function, then there is exactly one option. The game on a single heap reduces to a trivial she-loves-me-she-loves-me-not game, and in particular, the \SG-function reduces to a binary output, that is, $\SG(n)\in\{0,1\}$. Hence, such rulesets will be referred to as {\em binary rulesets}. The same game, however, played on several heaps, with at least one non-binary ruleset, can be highly non-trivial, and result in great complexity. The question of which heap to move on does not have a polynomial time solution in general, while many arithmetic functions are known to be intractable. Our inspiration for studying binary counting games came from Harold Shapiro's classification for the recurrence of the totient function \cite{HS}. A couple of examples will clarify these type of issues. 

\begin{example}\label{ex:div}
Let $0$ be the empty heap. Suppose that, from a heap of size $n>0$, the players can remove the number of divisors of $n$. The option of $n=1$ is $0$. A heap of size 2 also has $0$ as an option, but $1$ is the option of $3$. The \SG-sequence thus starts: $0,1,1,0,0,1$. The heap of size 5 has $3$ as the option, for which the nim-value is $0$. On one heap, while play is trivial, the problem of determining the winner is as hard as the complexity of the sequence.
\end{example}

\begin{example}\label{ex:propdiv}
Let $0$ be the empty heap. Suppose that, from a heap of size $n>0$, the players can move-to the number of proper divisors of $n$. The option of $n=1$ is $0$. The heaps of size two and three have moves to the heap with a single pebble. The number of proper divisors of $n=4$ is $2$, and hence the option is $2$. As for all primes, the option of $n = 5$ is the heap of size one. Thus, the \SG-sequence starts: $0,1,0,0,1,0$.      
\end{example}

\begin{example}\label{ex:3}
Consider binary games. Of course, even playing a disjunctive sum of binary games, gives only binary values. 
Consider, for example {\sc totient}, where $\SG(2+3+4+5)=\SG(2)\oplus \SG(3)\oplus \SG(4)\oplus \SG(5)=1\oplus 0\oplus 0 \oplus 1=0$. Hence $2+3+4+5$ is a second player winning position. To see this in play, suppose that the first player selects the heap of size $4$ and moves to $2+3+\phi(4)+5 = 2+3+2+5$. Now, $\SG(2+3+2+5)=1\oplus 0\oplus 1 \oplus 1=1$, which is a winning position for the player to move, and indeed, since every move changes the parity, we have automatic, `random' optimal play even if we play a sum of games, \emph{provided} that they are all binary. In particular if we play a disjunctive sum of totient games, then the optimal strategy is to play any move. Hence these games seem less interesting in that respect, as 2-player games, but suppose that we instead play a disjunctive sum of the totient game $G$ with the totative game $H$. Now an efficient algorithm for computing the binary value (see Theorem~\ref{thm:totient}) is interesting again. What is a sufficient move in the first player winning position $7_{{\rm totient}} + 7_{{\rm totative}}$? (There are exactly three winning moves.)
\end{example}

Those examples motivate play on arithmetic counting functions. Other examples of binary games are the {\sc fullset} games, where each move is defined by playing to a disjunctive sum of all numbers induced by the arithmetic function.

For the {\sc powerset} rulesets, the range of the opt function is the set of all subsets of natural numbers, a generic game on a single heap decomposes to play on several heaps. Hence, the full \SG-function is intrinsically motivated in the solution of a single game, even if the starting position is a single heap. 

\begin{example}\label{ex:4}
Consider the game played from a heap of size $n$, where the options are to play to any non-empty set of proper divisors of $n$. If $n=6$, then the options are the single heaps of size $1,2,3$ respectively, the pairs of heaps $1+2, 1+3, 2+3$, and the triple $1+2+3$. A heap of size one has no option, and a heap of size two or three has a heap of size one as option. Hence $\SG(6)=2$. The nim-value of each prime is one, and so on. 
\end{example}

\begin{example}\label{ex:5}
Consider the game played from a heap of size $n$, where the options are to play to any finite set of relatively prime residues smaller than $n$. If $n=5$, then the options are all nonempty subsets of $\{1,2,3,4\}$. In spite of the relatively large number of options, in this particular case, the \SG-computation becomes easy. A heap of size one has no option and so $\SG(1)=0$. Therefore, $\SG(2)=1$, and so $\SG(3)=2$. A heap of size $4$ has $1,3,1+3$ as options, and so $\SG(4)=1$. By this, obviously $\SG(5)=4$. A heap of size $6$ has few options, and easily $\SG(6)=1$. A heap of size $7$ has many options, and likewise easily $\SG(7)=8$, the smallest unused power of two. This game is revisited in Theorem~\ref{thm:powtotative}. 
\end{example}

In view of the above examples, we use the following classification of games on arithmetic functions; an  \emph{arithmetic game} satisfies one of these items.

\begin{itemize}
    \item[(i)] Play singletons from the arithmetic property;
\item[(ii)] Play the number of elements from the arithmetic property;
    \item[(iii)] Play the disjunctive sum of all numbers from the arithmetic property;
    \item[(iv)] Play any non-empty subset of numbers from the arithmetic property, as a disjunctive sum.
\end{itemize}

The word ``Play \ldots '' (read: ``Play is defined by\ldots '') is intentionally left open for interpretation. Here, it will have one out of two meanings; either the players move-to the numbers, or they subtract the numbers, from the given heap (size). The items (iii) and (iv) typically split a heap into several heaps to be played in a disjunctive sum of heaps. Note that (iii) is binary, although it does not concern counting functions. The rulesets induced by (iii) and (iv) above are not listed above, but naturally build on items 1, 2 and 4. We define them in their respective sections.

Some arithmetic functions directly induce a disjunctive sum of games, such as the division algorithm or the factoring problem. For the ruleset on Euclidian division from the first paragraph  (Section~\ref{sec:dividing}), we conjecture that the relative nim-values, $\SG(n)/n$, tend to 0 with increasing heap sizes.

In Section~\ref{sec:singleton}, we study singleton games. In Section~\ref{sec:counting}, we study counting games.  In Section~\ref{sec:dividingG}, we study dividing games, where division induces a disjunctive sum of games, and similar for Section~\ref{sec:factoring} with factoring games. In Section~\ref{sec:fullset}, we study disjunctive sum games on the full set induced by the arithmetic function. In Section~\ref{sec:powerset}, we study powerset disjunctive sum games.  Section~\ref{sec:disc} is devoted to some future direction.

For reference, let us include a table of studied rulesets, in the order of appearance, including some significant properties. The abbreviations are m-t: move-to, subtr.: subtraction, div.:divisor, rel.:relative, n.:number, pr.: problem, disj.: disjunctive. The solution functions are  defined in the respective sections, but let us list them here as well. In particular, we encounter {\em indexing functions}, where numbers with a certain property are enumerated, starting with $1$ for their smallest member, etc. In the table we find the following functions $\Omega$: number of prime divisors counted with multiplicity, $\omega$: the number of prime divisors counted without multiplicity, $\Omega_2$: the number of prime divisors counted with multiplicity, unless the divisor is 2, which is counted without multiplicity, $v$: usual 2-valuation, $i_o$: index of largest odd divisor, $i_p$: index of smallest prime divisor, \\

\vskip 10pt
\begin{tabular}{c|ccccc}
Ruleset & description&arithmetic f. & solution f. & Sec.\\\hline
{\sc maliquot}& m-t div. & aliquot & $\Omega$&\ref{sec:maliquot} \\
{\sc saliquot}& subtr. div. & aliquot & $v$& \ref{sec:saliquot} \\
{\sc maliquant}& m-t non-div. & aliquant & $i_o$& \ref{sec:maliquant} \\
{\sc saliquant}& subtr. non-div. & aliquant & partial sol.  &\ref{sec:saliquant} \\
{\sc totative}& m-t rel. prime & totative & $i_p$& \ref{sec:totative} \\
{\sc nontotative}& m-t nonrel. prime & totative & partial sol. &\ref{sec:nontotative} \\
{\sc totient}& m-t n. rel. prime & totient & Shapiro& \ref{sec:totient} \\
{\sc nontotient}& m-t num. nonrel. prime & totient  & &\ref{sec:nontotient} \\
{\sc mtau}& m-t n. div.  & $\tau$ &  observation&\ref{sec:mtau} \\
{\sc stau}& subtr. n. div. & $\tau$ & & \ref{sec:stau} \\
{\sc m}$\Omega$& m-t n. prime div. & $\Omega$ &  observation&\ref{sec:mOmega} \\
{\sc s}$\Omega$& subtr. n. prime div. & $\Omega$ &  &\ref{sec:sOmega} \\
{\sc m}$\omega$& m-t n. dist. prime div. & $\omega$ &  observation&\ref{sec:momega} \\
{\sc s}$\omega$& subtr. n. dist. prime div. & $\omega$ & &\ref{sec:somega} \\
{\sc dividing}&m-t disj. sum div.&aliquot&$\Omega_2$&\ref{sec:dividing}\\
{\sc div.-and-res.}&m-t disj. sum Eucl. div.& Eucl. div.&&\ref{sec:divrem}\\
{\sc compl.-grundy}&m-t disj. sum Eucl. div.& Eucl. div.&&\ref{sec:divrem}\\
{\sc div.-throw-res.}&m-t disj. sum Eucl. div.& Eucl. div.&$i_o$&\ref{sec:divrem}\\
{\sc res.-throw-div.}&m-t residue & Eucl. div.&yes&\ref{sec:divrem}\\
{\sc m-factoring} & m-t factoring & factoring &$ \Omega$&  \ref{sec:factoring}\\
{\sc s-factoring} & subtr. factoring & factoring &&  \ref{sec:factoring}\\
{\sc fs maliquot}& m-t disj. sum all div. & aliquot& square free  & \ref{sec:fullset} \\
{\sc ps maliquot}& m-t disj. sum div. & aliqout & &\ref{sec:powerset}\\
{\sc ps saliquot}& subtr. disj. sum div. & aliqout & $v$&\ref{sec:powerset}\\
{\sc ps maliquant}& m-t disj. sum div. & aliquant & $i_o$&\ref{sec:powerset}\\
{\sc ps saliquant}& subtr. disj. sum div. & aliquant & $i_p$&\ref{sec:powerset}\\
{\sc ps totative}& m-t disj. sum div. & totative & &\ref{sec:powerset}\\
{\sc ps nontotative}& m-t disj. sum div. & totative & &\ref{sec:powerset}\\
\end{tabular}\vskip 8pt\noindent

\section{Singletons}\label{sec:singleton}
This section concerns items 1,2 and 4 from the introduction, the aliquots, the aliquants and the totatives.

\subsection{The aliquots}\label{sec:aliquots}
The first game, {\sc maliquot}, is `{\sc nim} in disguise' (think of the prime factors of a number as the pebbles in a heap) but since the factoring problem is hard, the game is equally hard. Here, the arithmetic function is $f(n)=\{d : d\, |\, n, n\in \mathbb{N}_0\}$. In this section, the set of game positions is $\mathbb N$. Since all nonnegative integers divide 0, we do not admit 0 to the set of game positions. The second game, {\sc saliquot}, turns out to be somewhat more interesting. 

Let $n\in \N$. Then $\Omega (n)$ is the number of prime factors of $n$, counting multiplicities, and $v=v(n)$ is the 2-valuation of $n=2^vm$, where $m$ is odd. 

\subsubsection{{\sc maliquot}, move-to a proper divisor}\label{sec:maliquot}
Here, the set of move options from $n$ is $\mbox{opt}(n) = \{d : d\, |\, n, d\ne n, n\in \mathbb{N}\}$. 
\begin{example}
From 6 you have the options $1$, $2$ and $3$. From $7$ you may only move-to $1$.
\end{example}
The unique terminal position is $1$. It follows that $\SG(1)= 0$, and if $p$ is a prime then $\SG(p)=1$. We have computed the first few nim-values.
\vskip 8pt
\begin{tabular}{ccc}
$n$ & $\mbox{opt}(n)$ & $\SG(n)$\\\hline
1& $\varnothing$ & 0\\
2& 1 & 1\\
3& 1 & 1\\
4& 1,2 & 2\\
5& 1 & 1\\
6& 1,2,3 & 2\\
7& 1 & 1\\
8& 1,2,4 & 3
\end{tabular}\vskip 8pt\noindent
\begin{theorem}
Consider {\sc maliquot}. Then, for all $n$, $\SG(n) = \Omega(n)$. 
\end{theorem}
\begin{proof}
We have that $0=\SG(1)=\Omega (1)$, since there are no options from $1$, and $1$ does not have any prime factors. Suppose that $n>1$ has $k$ prime factors, counting multiplicities. Then, for each  $x\in \{1,\ldots ,k-1\}$, there is a divisor of $n$, corresponding to a move-to a number with $x$ prime factors. Since you are not allowed to divide by $n$, the number of prime factors decreases by moving, and so there is no option of nim-value $k$, by induction. By the mex-rule, the result holds, $\SG(n)=k=\Omega(n)$.
\end{proof}

\subsubsection{{\sc saliquot}, subtract a divisor}\label{sec:saliquot}
This game is defined on the nonnegative integers, $\mathbb{N}_0$. Here $$\mbox{opt}(n) = \{n-d: \; d\, |\, n, \, d>0\}.$$ 
\begin{example}
The options of $6$ are $5$, $4$, $3$ and $0$. From $7$ you can move to $0$ or $6$.
\end{example}
Since $0$ is always an option from $n$, it is clear that the nim-value of a non-zero position is greater than $0$. The initial nim-values are:
\vskip 8pt
\begin{tabular}{ccc}
$n$ & $\mbox{opt}(n)$ & $\SG(n)$\\\hline
0& $\varnothing$ & 0\\
1& 0 & 1\\
2& 0,1 & 2\\
3& 0,2 & 1\\
4& 0,2,3 & 3\\
5& 0,4 & 1\\
6& 0,3,4,5 & 2\\
7& 0,6 & 1\\
8& 0,4,6,7 & 4
\end{tabular}\vskip 8pt\noindent
As the table indicates, the nim-values concern the 2-valuation of $n$.
\begin{theorem}
Consider {\sc saliquot}. Then 
$\SG(0)=0$. Suppose that $n>0$ and let $n=2^{k}m$, where $2\nmid m$ and $k\ge 0$. Then $\SG(n)= v(n)+1=k+1$. 
\end{theorem}
\begin{proof}
The case $0\, |\, 0$ is excluded in the definition of opt. Hence there is no move from 0 and so $\SG(0)=0$. Note that if $n>0$ then $0$ of nim-value $0$ is an option.

Suppose that $n$ is odd. Then, for all $d\, | \, n$, $n-d$ is even. By induction the even numbers have nim-value greater than one. Hence, since 0 is an option, the mex-function gives $0+1=1$ as the nim-value of $n=2^{0}m$. 

Suppose that $n$ is even. Then $n-1$, which is odd, is an option. It has nim-value 1 by induction. (Hence, $\SG(n)\geqslant 2$.) Let $d=2^\ell q\leqslant n$, where we may assume that $\ell>0$, since we are interested in the even options. 

Since $d$ is a divisor of $n$, we have that $0\le \ell \leqslant k$, and $q\mid m$, with odd $m>1$. We get $n-d=2^{k}m-2^\ell q=2^\ell(2^{k-\ell}m-q)$. The number $x=2^{k-\ell}m-q$ is odd, of nim-value $1$ by induction, unless $k=\ell$.  In this case, if $m=q$, the option is $0$, so suppose $m>q$. Since both $m$ and $q$ are odd, then the option of $n$ has a greater $2$-valuation than $n$, i.e $v(n)\ge k+1$. Therefore, no option has $2$-valuation $k$, and hence by induction no option has nim-value $k+1$. 

Since $\ell$ can be chosen freely in the interval $0\le \ell< k$, by induction all nim-values $0\le \SG(n-d)\le k$ can be reached; since $m>1$, we may take $q<m$. The result follows.
\end{proof}

\subsection{The aliquants}
The aliquant games are somewhat more intricate than the aliquots, but we still have an explicit solution in the first variation. Here $f(n)=\{d : d \nmid n, n\in \mathbb{N}_0\}$. 

\subsubsection{{\sc maliquant}: move-to a non-divisor}\label{sec:maliquant}
Since all numbers divide $0$, $0$ does not have any options, and hence $\SG(0)=0$. On the other hand, $0$ does not divide any nonzero number, and hence $0$ will be an option from each number. The options are: $\opt(n)=\{d<n : d \nmid n, n\in \mathbb{N}_0\}.$
\vskip 8pt
\begin{tabular}{ccc}
$n$ & $\mbox{opt}(n)$ & $\SG(n)$\\\hline
0& $\varnothing$ & 0\\
1& 0 & 1\\
2& 0 & 1\\
3& 0,2 & 2\\
4& 0,3 & 1\\
5& 0,2,3,4 & 3\\
6& 0,4,5 & 2\\
7& 0,2,3,4,5,6 & 4\\
8& 0,3,5,6,7 & 1
\end{tabular}\vskip 8pt\noindent
In {\sc maliquant}, the 2-valuation plays an opposite role as in {\sc saliquot}; here only the odd part of $n$ determines the nim-value. Let $i_o:\N\rightarrow \N$ be the index function for largest odd factor of a given natural number. That is, if $n=2^k(2m-1)$, then $i(n) = m$. Clearly $i_o(2m-1)=m$ and $i_o(n)\le(n+1)/2$.

\begin{lemma}\label{lem:oddind}
For all $n\in \N$, the numbers in the set $\{n,\ldots ,2n-1\}$ contribute all maximal odd factor indices in the set $\{1,\ldots , n\}$. That is, $\{i_o(x) : n\le x\le 2n-1\}=\{1,\ldots ,n\}$.
\end{lemma}
\begin{proof}
We use induction on~$n$. Assuming the statement of the lemma holds for all natural numbers up to~$n$, we consider the set $\{n+1, \dots, 2n-1, 2n, 2n+1\}$. As $i_o(2n)=i_o(n)$ and $i_o(2n+1)=n+1$, it follows that $\{i_o(x) : n\le x\le 2n+1\}=\{1,\ldots ,n, n+1\}$
\end{proof}
\begin{theorem}\label{thm:maliquant}
Consider {\sc maliquant}. Then, $\SG(0)=0$ and, for all $n\in \N$, $\SG(n)=i_o(n)$.
\end{theorem}
\begin{proof}
We use induction on~$n\in \N$. If~$n$ is odd, then write $n=2m-1$ whence $\{m-1,\dots,2m-3\}\subset\mbox{opt}(n)$. By Lemma~\ref{lem:oddind} and induction hypothesis, we have $\{\SG(k):m-1\le k\le2m-3\}=\{1,\dots,m-1\}$. By induction hypothesis, $\SG(k)=i_o(k)\le m/2$ for all $k<m-1$, and hence $\SG(n)= m=i_o(n)$. Even numbers are options, but they have nim-values smaller than $m/2$, by induction.

If~$n$ is even, write $n=2^km$ where $k>0$ and~$m$ is odd. Thus it suffices to prove $\SG(n)=i_o(m)$. Note that, for all positive integers~$l$, we have $i_o(l)=i_o(m)$ if and only if $l=2^jm$ for some $j\ge0$. Thus if $m=1$ then $\mbox{opt}(n)$ omits all those elements $2^j$ with index~$1$, hence $\SG(n)=1$. Otherwise, if  $m>1$, we observe that
$$\left\{2^{k-1}m+1,2^{k-1}m+2,\dots,2^km-1\right\}\subset\mbox{opt}(n).$$
If we augment this set with the element $2^km$, then by Lemma~\ref{lem:oddind} the indices of its elements are $\{1,2,\dots,2^{k-1}m+1\}$, which includes $\{1,2\dots,i_o(m)-1\}$ as a subset. However, $\mbox{opt}(n)$ contains no elements of the form $2^jm$, and hence $i_o(m)$ does not appear among the elements $i_o(x)$ for all $x\in\mbox{opt}(n)$, yet all elements of $\{1,\dots,i_o(m)-1\}$ do so appear. Thus by induction hypothesis $\SG(n)=i_o(m)$. 
\end{proof}

\subsubsection{{\sc saliquant}: subtract a non-divisor}\label{sec:saliquant}
Here we run into some mysterious sequences. We can only prove partial results. The options are: $\opt(n)=\{n-d : d \nmid n, n\in \mathbb{N}_0\}.$

\iffalse
 \begin{verbatim}
restart; with(numtheory); N := 100

mex := proc (a::list)::integer; 
local i::integer; 
for i to max(op(a))+1 do if i in a then  else break end if end do; i 
end proc

f[1] := 0; f[2] := 0; 
for n from 3 to N do a := {op(divisors(n))}; 
b := `minus`({seq(i, i = 1 .. n)}, a); 
c := map(proc (x) options operator, arrow; n-x end proc, b); 
f[n] := mex([op(map(proc (x) options operator, arrow; f[x] end proc, c))]) 
end do

\end{verbatim}
\fi

\vskip 8pt
\begin{tabular}{cccccc}
$n$ & $\mbox{opt}(n)$ & $\SG(n)$&$n$ & $\mbox{opt}(n)$ & $\SG(n)$\\\hline
0& $\varnothing$ & 0 &10& 1,2,3,4,6,7& 2\\
1& $\varnothing$ & 0&11&1,2,3,4,5,6,7,8,9&5\\
2& $\varnothing$ & 0&12&1,2,3,4,5,7&4\\
3& 1 & 1&13&$1,\ldots ,11$& 6\\
4& 1 & 1&14&$1,\ldots , 6$, 8,9,10,11&6\\
5& 1,2,3 & 2&15&$1,\ldots , 9$,11,13 &7\\
6& 1,2 & 1&16&$1,\ldots , 7$, 9,10,11,13 & 7\\
7& 1,2,3,4,5 & 3&17&1,\ldots ,15& 8\\
8& 1,2,3,5 & 3&18&1,\ldots , 8,10,11,13,14&4\\
9& 1,2,3,4,5,7 & 4&19&1,\ldots ,17&9
\end{tabular}\vskip 8pt\noindent

The odd heap sizes turn out to be simple. We give some more nim-values for even heap sizes, $n=0,2,\ldots$, 
$$\SG(n)=0,0,1,1,3,2,4,6,7,4,7,5,10,12,10,13,15,8,13,9,17,17,16,11,22,\ldots$$ 
For even heap sizes $n\ge 2$, 
$$\SG(n)/n=0,1/4,1/6,3/8,1/5,1/3,3/7,7/16,2/9,7/20,5/22,5/12,6/13,5/14,$$$$13/30,15/32,4/17,13/36,9/38,17/40,17/42,4/11,11/46,11/24,\ldots$$

Sorting the ratios $\SG(n)/n$ by size, we find that the associated nim-values $[n, \SG(n)]$ for the smallest ratios, $[6,1], [10,2], [18,4], [22,5], [34, 8], [38,9], [46,11]$ satisfy $(n-2)/\SG(n)=4$. 
The half of each heap size in this sequence is odd, and we get the odd numbers $3,5,9,11, 17, 19, 23, \ldots$. We have not investigated these patterns further, but we believe that, for all $n$, $\SG(n)\ge (n-2)/4$. Indeed, by plotting the first 1000 nim-values in Figure~\ref{fig:saliquant}, this lower bound appears to continue.\\

\begin{figure}[h!]
  \centering
    \includegraphics[width=0.6\textwidth]{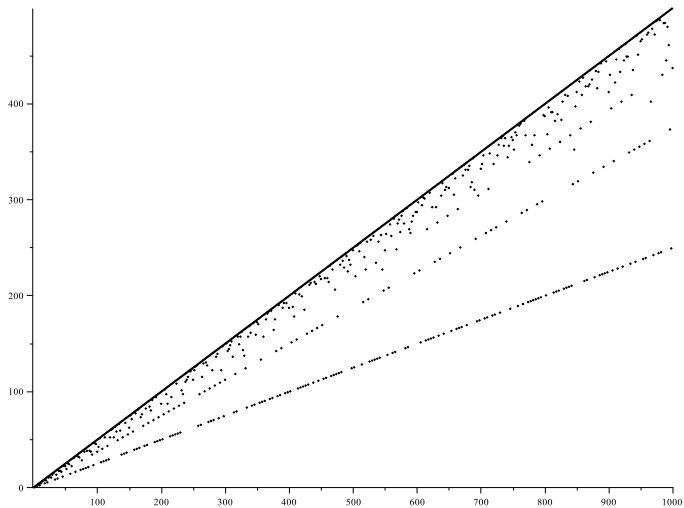}
      \caption{The initial 1000 nim-values of {\sc saliquant}.}\label{fig:saliquant}
\end{figure}

\begin{theorem}
Consider {\sc saliquant}. Then $\SG(0)=0$, and if $n$ is odd, then $\SG(n)=\frac{n-1}{2}$. Moreover,  $\SG(n)<n/2$. 
\end{theorem}
\begin{proof}
Suppose that the statement holds for all $m<n$. If $n=2x+1$ then each nonnegative integer smaller than $x$ is represented as a nim-value, and specifically, for each odd number $2y+1$, with $y<x$, $\SG(2y+1)=y$. Moreover, each odd number is an option of $x$, since each even integer is a non-divisor of $n=2x+1$. Therefore, we use that, by induction, each even number smaller than $n$ has a smaller nim-value and we are done with the first part of the proof. 

Suppose next that $n=2x$. Then, since both 1 and 2 are divisors, we know that the largest option is smaller than $2x-2$. By induction, the nim-value of any number smaller than $2x-2$ is smaller than $x-2$, namely $\SG(2x-3)=\SG(2(x-2)+1)$ is the upper bound for a nim-value of an option of $2x$. 
\end{proof}

\subsection{The totatives: move-to a relatively prime}\label{sec:totative}
Here $f(n) = \{x\mid (x,n)=1\}$. The totative games are defined by moving to a relatively prime residue. We list the first few nim-values of {\sc totative}. 
\vskip 8pt
\begin{tabular}{ccc}
$n$ & $\mbox{opt}(n)$ & $\SG(n)$\\\hline
1& $\varnothing$ & 0\\
2& 1 & 1\\
3& 1,2 & 2\\
4& 1,3 & 1\\
5& 1,2,3,4 & 3\\
6& 1,5 & 1\\
7& 1, \ldots , 6 & 4\\
8& 1,3,5,7 & 1
\end{tabular}\vskip 8pt\noindent
We have the following result. The solution involves the function $i_p$, the index of the smallest prime divisor of a given number, where the prime 2 has index 1.

\begin{theorem}
Consider {\sc totative}. The nim-value of $n>1$ is the index of the smallest prime divisor of $n$, and $\SG(1)=0$.
\end{theorem}
\begin{proof}
There is no move from $1$, because the only number relatively prime with $1$ is $1$, and options have smaller size than the number. Hence, by the definition of the mex-function, $\SG(1)=0$. Also, $\SG(2)=1$, since the only relatively prime number of $2$ is $1$, which has nim-value $0$, and $i_p(2)=1$. Suppose that the result holds for all numbers smaller than $n$. From $n$ you can only access a smaller number with no common divisor to $n$. Therefore none of its options has the same smallest prime divisor. This is one of the properties of the mex-rule. 

Thus, the index of the smallest prime divisor of $n$ will be chosen as nim-value if each prime with a smaller index appears as an option. 
But, the set of relatively prime numbers smaller than $n$ contains in particular all the relatively prime numbers of the smallest prime divisor of $n$, and hence, all the primes that are smaller than the smallest prime factor of $n$. By induction, this is the desired set of nim-values, since the move-to 1 (of nim-value $0$) is always available. 
\end{proof}

This is sequence A055396 in Sloane~\cite{Sloane}: ``Smallest prime dividing $n$ is $a(n)$-th prime $(a(1)=0)$."

\subsection{The non-totatives: move-to a non-relatively prime}\label{sec:nontotative}
Here is a table of the first few nim-values of {\sc nontotative}:
\vskip 8pt
\begin{tabular}{cccccc}
$n$ & $\mbox{opt}(n)$ & $\SG(n)$&$n$ & $\mbox{opt}(n)$ & $\SG(n)$\\\hline
0& $\varnothing$ & 0 &10& 0,2,4,5,6,8& 5\\
1& 0 & 1&11&0&1\\
2& 0 & 1&12&0,2,3,4,6,8,9,10&6\\
3& 0 & 1&13&$0$& 1\\
4& 0,2 & 2&14&0,2,4,6,7,8,10,12,&7\\
5& 0 & 1&15&0,3,5,6,9,10,12 &4\\
6& 0,2,3,4 & 3&16&0,2,4,6,8,10,12,14 & 8\\
7& 0 & 1&17&0& 1\\
8& 0,2,4,6 & 4&18&0,2,3,4,6,8,9,10,12,14,15,16&9\\
9& 0,3,6 & 2&19&0&1
\end{tabular}\vskip 8pt\noindent

 This sequence does not yet appear in OEIS~\cite{Sloane}, but curiously enough, a nearby sequence is A078898, ``Number of times the smallest prime factor of $n$ is the smallest prime factor for numbers $\le n$; $a(0)=0$, $a(1)=1$.'' For $n \ge 2$, $a(n)$ tells in which column of the sieve of Eratosthenes (see A083140, A083221) $n$ occurs in. %From OEIS: $a(n)=0, 1, 1, 1, 2, 1, 3, 1, 4, 2, 5, 1, 6, 1, 7, 3, 8, 1, 9, 1, 10, 4, 11, 1, 12, 2, 13, 5, 14, 1, 15, 1, 16, 6, \ldots$. 
%$17, 3, 18, 1, 19, 7, 20, 1, 21, 1, 22, 8, 23, 1, 24, 2, 25, 9, 26, 1, 27, 4, 28, 10, 29, 1, 30, 1, 31, 11,$ $ 32, 5, 33, 1, 34, 12, 35, 1, 36, 1, 37, 13, 38, 3, 39, 1, 40, 14, 41, 1, 42, 6, 43,\ldots $
Here, $\SG(15)=4\ne 3=a(15)$ is the first differing entry. In Figure~\ref{fig:nontotative} we plot the first $1000$ nim-values.\\%The first 34 entries are 
%$0, 1, 1, 1, 2, 1, 3, 1, 4, 2, 5, 1, 6, 1, 7, 4, 8, 1, 9, 1, 10, 5, 11, 1, 12, 2, 13, 7, 14, 1, 15, 1, 16, 8$. %The beginning up to $0, 1, 1, 1, 2, 1, 3, 1, 4, 2, 5, 1, 6, 1, 7, 4, 8$ does not match any other sequence in SLOANE. 
%Here is a sequence until SG 100 appears obtained by maple code:  
%$0,1,1,1,2,1,[6,3],1,[8,4],2,[10,5],1,[12,6],1,[14,7],[15,4],8,1,9,1,[20,10],[21,5],[22,11],1,12,[25,2],13,[27,7],14,1,15,$
%$1,16,[33,8],17,[35,3],18,1,19,[39,10],20,1,21,1,22,[45,11],23,1,[48,24],[49,2],25,[51,13],26,1,27,[55,6],28,[57,14],29,1,30,1,31,[63,16],32,$
%$[65,7],33,1,34,[69,17],35,1,36,1,37,[75,19],38,[77,4],39,1,40,[81,20],41,1,42,[85,8],43,[87,22],44,1,45,[91,6],46,[93,23],47,[95,9],48,1,49,[99,25],$
%$50,1,51,1,52,26,53,1,54,1,55,28,56,1,57,12,58,29,59,9,60,[121,2],61,31,62,13,63,1,64,32,65,1,66,8,67,$
%$34,68,1,69,1,70,35,71,[143,3],72,14,73,37,74,1,75,1,76,38,77,16,78,1,79,40,80,10,81,1,82,41,83,1,84,2,85,$
%$43,86,1,87,17,88,44,89,1,90,1,91,46,92,18,93,5,94,47,95,1,96,1,97,49,98,1,99,1,100.$
%\fi
\begin{figure}[h!]
  \centering
    \includegraphics[width=0.6\textwidth]{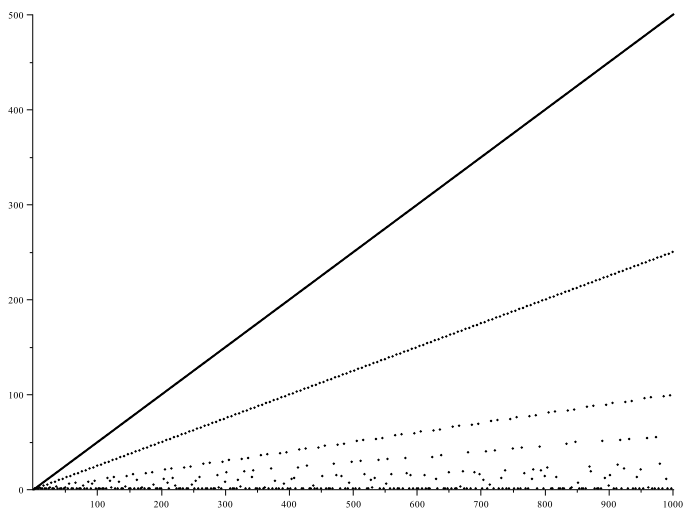}
    \caption{The initial $1000$ nim-values of {\sc nontotative}.}\label{fig:nontotative}
\end{figure}

Let us sketch a few nim-value subsequences. The primes have nim-value $1$, and the prime squares have nim-value $2$. The numbers with close prime factors, `almost squares', appear to have almost constant nim-values. On the other hand, some numbers in arithmetic progressions appear to have nim-values in almost arithmetic progressions. 
For all $n$, $\SG(2n)=n$.   We give the exact statements in Theorem~\ref{thm:nontotative} below. 
 
 For subsequences of the natural numbers, $s$, let the asymptotic relative nim-value be $$r_s=\lim_n\frac{\SG(s(n))}{s(n)},$$ if it exists. The subsequences of largest relative nim-values, apart from $s_0 =2,4,\ldots$ (with $\SG(2n) = n$), are  $s_1=3,9,\ldots$ and $s_2=5,25,35,55,65,\ldots$, with corresponding first differences $\Delta_1 = (6, 6, \ldots)$ and  $\Delta_2 = (20, 10, 20, 10,  \ldots )$, with nim-values in  $\{\lfloor (n+1)/4 \rfloor\}$ and $\{\lfloor n/10\rfloor, \lceil n/10\rceil\}$ respectively.  Thus $r_0=1/2$, $r_1=1/4$ and $r_2=1/10$, where $r_{s_i}=r_i$. The region between `prime factorization' and `purely arithmetic behavior' is still mysterious. We can identify at least one more sequence of arithmetic behavior, with $r_3\approx 1/17$, but the descriptions start to get quite technical here. Note that (ix) and (v) imply (viii); it is handy to state (v) as a separate item as it is used several places in the proof. 
 
\begin{theorem}\label{thm:nontotative}
Consider {\sc nontotative}. For $n\in \N$, 
\begin{enumerate}[(i)]
\item $\SG (n) = 1$ if and only if $n$ is prime; 
\item $\SG (n) = 2$ if and only if $n$ is a prime square; 
\item $\SG(n)\in \{3, 4\}$ if and only if $n = p_ip_{i+1}$, or $n=8$: $\SG(n)=3$ if and only if $i$ is odd, with $p_1=2$; 
\item $\SG(n)\in \{5, 6\}$ if and only if $n = p_ip_{i+2}$ or $n=12$: $\SG(n)=5$ if and only if $i\equiv 1,2\pmod {4}$. 
\end{enumerate}
Moreover, for $n\in \N$, 
\begin{enumerate}
\item[(v)] $\SG(2n)=n$; 
\item[(vi)] If $n\equiv 3\pmod 6$, then  $\SG(n)=\lfloor (n+1)/4 \rfloor$; 
\item[(vii)] if $n\equiv 5,25\pmod{30}$, then  $\SG(n)\in \{\lfloor n/10\rfloor,\lceil n/10\rceil\}$. 
\end{enumerate}
Lastly, for $n\in \N$, 
\begin{enumerate}
\item[(viii)] $\SG(n)\le n/2$; 
\item[(ix)] if $n$ is odd, then $\SG(n)\le (n+1)/4$.
\end{enumerate}
\end{theorem}
\begin{proof}
Note that $\SG(0)=0$ implies $\SG(n)>0$ if $n>0$. The induction hypothesis assumes all the items. Not that item (v) takes care of all cases where the prime $2$ divides $n$, so in all other items, we may assume that the smallest prime dividing $n$ is greater than $2$.

For (i), the only non-relatively prime number of a prime is 0. Hence $\SG(p)=1$ if $p$ is prime. If $n=pm$ is not a prime, then there is a move to the prime divisor $p$, a non-relative prime, and there is a move to 0. Hence the nim-value is greater than one. 

For (ii) we consider prime squares $p^2$, and note that each option is of the form $np$, $0\le n\le p-1$. In particular there are moves $p^2\mapsto 0$ and $p^2\mapsto p$, as noted in the first paragraph. Moreover, by induction we assume that $\SG(m) = 2$ if and only if $1<m<p^2$ is a prime square. Then $m\ne np$, and so, by the minimal exclusive algorithm, $\SG(p^2) = 2$. For the other direction, we are done with the cases $0,1$ and primes. Consider the composite $n=pm$, not a prime square, where $p$ is the smallest prime factor. Then there is a move to $p^2$, and hence $\SG(n)\ne 2$.

For (iii) we begin by proving that $\SG(n)\in \{3, 4\}$ if $n = p_ip_{i+1}$ or $n=8$, and where the nim-value is three if and only if the index of the smaller prime is odd. The base case is $\SG(2\cdot 3)=3$, and where the exception $n=8$ is by inspection. For the generic case, each option is of the form $np_i$, $0\le n\le p_{i+1}-1$ or $np_{i+1}$, $0\le n\le p_{i}-1$. In particular, there is a move to $p_i$ (and to $p_{i+1}$) of nim-value one, and there is a move to $p_i^2$ of nim-value $2$. We must show that there is no move to another prime pair of the same form, i.e., $p_jp_{j+1}$, with $j$ of the same parity as $i$. Observe that there is a move to $p_{i-1}p_i$, with $j+1=i$, but there is no move to any other almost square $p_jp_{j+1}$. By induction, this observation suffices to find a move to nim-value $3$, if $i$ is even. For the other direction, we must show that $\SG(n)\not \in\{3,4\}$ if $n$ is not an almost square. We are done with the cases $n$ a prime or a prime square. Suppose that $n=px$ is not of the mentioned form, where $p$ is the smallest prime factor of $n$. The case $p=2$ is dealt with in item (v), so lets assume $p>2$. Then, there is a move to $pq$ (non-relative prime with $n$), where $q$ is the smallest prime larger than $p$, because by assumption, $x>q$. And there is a move to  $pq$, where $q$ is the largest prime smaller than $p$.

For (iv), we study the case $n = p_ip_{i+2}$. If $i=1$, then $n=10$, and $\SG(10)=5$. If $i=2$, then $n=21$, and, by inspection,  $\SG(21)=5$. For the general case, among the options we find $0,p_i,p_i^2, p_ip_{i+1},p_{i+1}p_{i+2}$. Hence, by the previous paragraphs, the options attain all nim-values smaller than 5. Next, suppose that $i\equiv 1,2\pmod 4$, and we must show that there is no option of the same form, to create nim-value $5$. Each option is a multiple of one of the primes $p_i$ and $p_{i+2}$. The only possibility would be the option $p_{i-2}p_i$. But $i-2\equiv 0,3 \pmod 4$. Hence no option has nim-value $5$. The analogous argument suffices to show that no option has nim-value $6$ if $i\equiv 0,3\pmod 4$, $i\ge 3$. The special case $n=12=2\times 2\times 3 $ is not an option, by $i\ge 3$. On the other hand, the argument shows that there is an option to nim-value~$5$. Consider the other direction. Suppose that $n=p_ix$, where $p_i>2$ is the smallest prime in the factorization of $n$. (The case $p=2$ is dealt with below.) If $x>p_{i+2}$ then there is a move to $p_ip_{i+2}$, and if $i\ge 3$, then there is a move to $p_{i-2}p_i$.  If $i=2$, then there is a move to $12$ of nim-value $6$. That concludes this case. If $p_i<x<p_{i+2}$, then $x=p_{i+1}$, since $p_i$ is the smallest prime in the decomposition of $n$, and we are done with this case.

For (v), we verify that, for all $n$, $\SG(2n)=n$. The options are of the forms $2j$, with $0\le j\le n$, and so, induction on (v) gives that each nim-value  smaller than $n$ can be reached. Moreover induction on (viii) gives that nim-value $n$ does not appear among the options. 

For (vi) we must prove: if $n\equiv 3\pmod 6$, then $\SG(n)=\lfloor (n+1)/4 \rfloor$. The claimed nim-value sequence for the positions $3,9,15,21,27,\ldots$ is $\lfloor (3+1)/4 \rfloor, \lfloor (9+1)/4 \rfloor,\lfloor (15+1)/4 \rfloor,\ldots$, which is $1,2, 4, 5, 7,8,\ldots$. Clearly $n$ has each smaller position of the same form as an option. Precisely, the multiples of 3 are missing in the nim-value sequence. But, induction on (v), the multiples of $6$ have nim-values multiples of $3$, and indeed, all multiples of $6$, smaller than $n$ are options of $n$. By induction on item (ix), since $n$ is odd, the nim-value $\lfloor (n+1)/4 \rfloor$ does not appear among its options. The proof of (vii) is similar to (vi), but more technical, so we omit it.

Item (viii) follows directly by induction (for example, if $n$ is even then $n-2$ is the largest option and $\SG(n-2)\le n/2-1$). 

For item (ix), assume that $p>2$ is smallest prime divisor of $n$. The cases with $p\le 5$ have already been proved in items (vi) and (vii). Hence $p>5$. It follows that with $t=\lfloor \frac{n}{2p}\rfloor$, $tp+3< n/2<(t+1)p-3$. It follows that the nim-value $\lfloor \frac{n+1}{4}\rfloor$ cannot be reached from $n$, by options of the form in (v). On the other hand it cannot be reached, by moving to an odd number, since options $n-2p$ or smaller produce too small nim-values, by induction.
\end{proof}

We do not yet know, if all nim-values can be obtained by analogous reasoning. The initial occurrences (as for $n=12$ in the proof above) of nim-values that do not follow general patterns may complicate things.

\section{Counting games}\label{sec:counting}
This section concerns rulesets as in item (i) in the introduction. \emph{Binary games} have only one option per heap. At each stage of play, the decision problem reduces to which one of the heaps to make the move.  
The nim-value of any sum of binary games is binary, that is, each nim-value $\in \{0,1\}$. Indeed, the nim-value of a given disjunctive sum of binary games is 0 if and only if the number of heaps of nim-value one is even. Of course, the nim-value sequence for any given ruleset is valid in the much larger context of all normal-play combinatorial games.  

We begin in Section~\ref{sec:totient}, by solving {\sc totient}, and then we sketch a classification scheme for {\sc nontotient}. Then we  list nim-values of other open counting games. 

\subsection{Harold Shapiro and the totient games}\label{sec:HS}
Recall that Euler's totient (or $\phi$) function counts the number of relatively prime residues of a given number. For example $\phi(7)=6$ and $\phi(6)=2$. Recall that this function is multiplicative. 
This is where we can apply a known result by Harold Shapiro ``An arithmetic function arising from the $\phi$ function" \cite{HS}.  The fundamental theorem for iteration of the Euler $\phi$ function in his work is as follows. For all $x$, let $\phi^i(n)=\phi^{i-1}(\phi(n))$. Since $\phi^i(n) < \phi^{i-1}(n)$ and $\phi(n)$ is even, if $n>2$, we have that, for all $n$ and some unique $i>0$ (depending on $n$ only)
\begin{align}\label{class}
\phi^i(n)=2. 
\end{align}
 This lets us define the Class of $n$, as $C(n)=i$, when (\ref{class}) holds, and otherwise $C(1)=C(2)=0$.

\begin{theorem}[\cite{HS}]\label{thm:C}
Let $m,n\in \N$. If  $n$  is odd, then $C(n) = C(2n)$, and otherwise $C(n)+1 = C(2n)$. In general, if either $m$ or $n$ is odd, then $C(mn)=C(m)+C(n)$. Otherwise, that is, if both $m$ and $n$ are even, then $C(mn)=C(m)+C(n)+1$. 
\end{theorem}
For example $\phi(7)^2 = \phi(6) = 2$, so $C(7)=2$. In general, for primes $p$, $\phi(p)^2=\phi(p-1)$, so $C(p)=C(p-1)+1$. Here is an example: $C(15)=C(3)+C(5)=1+C(4)+1=2+C(2) +C(2)+1=3$, and note that $\phi(15) = \phi(3)\phi(5)=2\cdot 4=8$. Moreover $\phi(8)=4$ and $\phi(4)=2$, so indeed $\phi^3(15)=2$.

This result lets us compute the nim-values of the first totient ruleset, {\sc totient}, simply by recalling the $\phi$-values for the primes.

\subsubsection{{\sc totient}, the move-to variation of the totient game}\label{sec:totient}
As mentioned, the nim-value table contains  only $0$s and $1$s, and we call this kind of games `binary games'.  Indeed, there is only one option from $n$, namely $\phi(n)$, the number of relatively prime residues of $n$. 
\vskip 8pt
\begin{tabular}{ccc}
$n$ & $\mbox{opt}(n)$ & $\SG(n)$\\\hline
1& $\varnothing$ & 0\\
2& 1 & 1\\
3& 2 & 0\\
4& 2 & 0\\
5& 4 & 1\\
6& 2 & 0\\
7& 6  & 1\\
8& 4 & 1
\end{tabular}\vskip 8pt\noindent

Here are a few more nim-values in the \SG-sequence: $$01001,01111,01010,01100,00101,0001$$ where the ``,'' is for readability. Each game component has a forced move, that if played alone may be regarded as an automaton. Starting from 8, for example, the iteration of $\phi$ gives the sequence of moves $8\mapsto 4\mapsto 2\mapsto 1$, and the $\SG$-sequence, of course is alternating between 0s and 1s, terminating with the 0 at position 1. If played on a disjunctive sum of {\sc totient}, the nim-value sequence is of course also binary, alternating between 0s and 1s, and it is 0 if and only if there is an even number of heaps of nim-value 1. Suppose that we play $7_{\rm t}+7_{\rm s}$, where the first $7$ is {\sc totient} and the second $7$ is {\sc subtraction}$\{1,2\}$, with nim-value sequence $0,1,2,0,1,2,0,\ldots$, say with sink number $1$ on both components. Then there are exactly two winning moves to $6_{\rm t}+7_{\rm s}$ or $7_{\rm t}+5_{\rm s}$. Intelligent play from a general position $m_{\rm t}+n_{\rm s}$ requires full understanding of {\sc totient}.

\begin{theorem}\label{thm:totient}
Consider {\sc totient}, and let $C$ be as in Theorem~\ref{thm:C}. Then $\SG(1)=0$, and for $n>1$, $\SG(n)=C(n)+1\pmod 2$.
\end{theorem}
\begin{proof}
Use Theorem~\ref{thm:C}.
\end{proof}

In general, thus it suffices to compute the parity of $C(n)$ and, given the factorization of $n$, apply Theorem~\ref{thm:C}. For example, $C(2^3)=C(2)+C(2^2)+1=3C(2)+2=2$, and without looking into the table, we get $\SG(8)= 1$. For another example, if $n=2\cdot 3^7\cdot 11$, then $C(n)= 7C(3)+C(11)=7 + 3=10$, since $\phi(3)=2$ and $\phi^3(11)=2$. Therefore $\SG(48114)=(C(48114)+1)\pmod 2=11\pmod 2=1$, and we find a unique winning move  $48114_{\rm t}+3_{\rm s}\mapsto 48114_{\rm t}+2_{\rm s}$, where {\rm s} is still {\sc subtraction}$\{1,2\}$. 

\subsubsection{{\sc nontotient}, the subtraction variation of the totient game.}\label{sec:nontotient}
From a given number $n$, subtract the number of relatively prime numbers smaller than $n$. We cannot adapt Theorem~\ref{thm:totient}, because it relies on iterations where you instead move-to this number, and the authors have not yet found a similarly efficient tool. Let us list the initial options and nim-values. An alternative way to think of the options is: move-to the number of nonrelative primes, including the number. The nim-values alternate for heaps that are powers of primes, starting with $\SG(p^0)=0$. This happens, because the number of nonrelatively prime numbers smaller than or equal a prime power $p^k$ is $p^{k-1}$. Hence, $\SG(p^{k})=0$ if and only if $k$ is even. Since $\phi$ is multiplicative it is easy to compute $f(n)=n-\phi(n)$, for any $n$, or get a formula for $f$, for any given prime decomposition of $n$. However, $f$ is not multiplicative, which limits the applicability of such formulas. In some special cases, we can use the proximity to powers of primes for fast computation of the nim-value. Take the case of $n=p^kq$, for some distinct primes $p$ and $q$. Then $f(n)=n-\phi(n)=p^{k-1}(q +p-1)$. Whenever $q+p-1$ is a power of the prime $p$, the nim-value of $n$ is immediate by the parity of the new exponent. Take for example $p=2$ and $q=7$. Then $p+q-1=2^3$, and so we can find the nim-value of, for example $n=7168=2^{10}\times 7$ gives the exponent $9+3=12$ and so $\SG(7168)=0$. Similarly, with $p=3$ and $q=7$, we can easily compute $\SG(413343)=1$, because $p+q-1=3^2$, and $413343=3^{10}\times 7$.

Let `$\mathrm{dist}$' denote the number of iterations of $f$ to an even power of a prime. We get the following suggestive table of the first few nim-values. We leave a further classification of  $\mathrm{dist}$ as an open problem.

\vskip 8pt
\begin{tabular}{cccc}
$n$ & $\mbox{opt}(n)$ & {\rm dist} & $\SG(n)$\\\hline
1& $\varnothing$ &0& 0\\
2& 1 &1& 1\\
3& 1 &1& 1\\
4& 2 &0& 0\\
5& 1 &1& 1\\
6& 4 &1& 1\\
7& 1  &1& 1\\
8& 4 &1& 1\\
9&3&0&0\\
10&6&2&0\\
11&1&1&1\\
12&8&2&0\\
13&1&1&1\\
14&8&2&0\\
15&7&2&0\\
16&8&0&0
\end{tabular}\vskip 8pt\noindent

\subsection{The $\tau$-games}
The \SG-sequences of {\sc mtau} and {\sc stau} do not yet appear in OEIS. The number of divisors is multiplicative in the following sense: $\tau(n)=(a_1+1)\cdots (a_k+1)$, where $n=p^{a_1}_1\cdots p^{a_k}_k$.

\subsubsection{Move-to the number of proper divisors}\label{sec:mtau}
Consider {\sc mtau}, where the single option is the number of proper divisors.\footnote{Note that if we remove the word proper here, then both 1 and 2 become loopy, and thus all games would be drawn. See also Section~\ref{sec:disc} for some more reflections on `loopy' or `cyclic' games.} A heap of size one has no option, so the nim-value sequence starts at $\SG(1)=0$. Let us list the first few nim-values.

\vskip 8pt
\begin{tabular}{ccc}
$n$ & $\mbox{opt}(n)$ & $\SG(n)$\\\hline
%0& $\varnothing$ & 0\\
1& $\varnothing$ & 0\\
2& 1 & 1\\
3& 1 & 1\\
4& 2 & 0\\
5& 1 & 1\\
6& 3 & 0\\
7& 1 & 1\\
8& 3 & 0\\
9& 2 & 0
\end{tabular}\vskip 8pt\noindent

Note that each prime has nim-value $1$ because they have only one proper divisor. From this small table we may deduce many more nim-values. The first few 0-positions are $$1, 4, 6, 8, 9, 10, 12, 14, 15, 18, 20, 21, 22, 24, 25, 26, 27, 28, 30,\ldots $$ Note that $16$ is the first composite number that is not included, and $36$ is the second one, and then $48, 80,81,100,$ etc. What is special about these composite numbers? 

The sequence of all ones has some resemblance to the sequence of all numbers with a nonprime number of proper divisors. As mentioned, $16$ and $36$ are the first composite members of this sequence. These two numbers are the smallest composite numbers with a composite, i.e., nonprime,  number of proper divisors, such numbers generalize the primes, because primes also have a nonprime number of proper divisors. We are interested in the smallest number $n=p^{a_1}_1\cdots p^{a_k}_k$, for which $\tau(n)-1=(a_1+1)\cdots (a_k+1)-1\in \{16,36,48, 80,\ldots\}$, that is, the smallest number $n$, such that $(a_1+1)\cdots (a_k+1)\in \{17,37,49, 81,\ldots\}$. An obvious candidate is $n=2^{16}$, with $a_1=16$, and otherwise $a_i=0$. But it turns out that $n=2^6\cdot 3^6=46656<2^{16}$ gives $\tau(46656)=(6+1)(6+1)=49$, and this is indeed the smallest such number. Thus, we have the following observation. 

\begin{observation}
Consider {\sc mtau}. If $n< 46656$, then $\SG(n)=1$ if and only if $n$ consists of a nonprime number of divisors. 
 \end{observation}

 We note that neither sequence is listed in OEIS (see Section~\ref{sec:mOmega} for a similar sequence that is listed). 

\subsubsection{Subtract the number of divisors}\label{sec:stau}
Consider {\sc stau}. This variation has a \SG-sequence beginning with $\SG(0)=0$. A heap of size one has one divisor, with an option to zero. A heap of size two has two divisors and hence the option is zero, and so on: $0,1,1,0,0,1,0,0,1,1,1,0,1,1,0,1,1,0,0,1,1,1$.
The `1's occur at $1,2,5,8,9,10,12,13,15,16,19,20\ldots$

\subsection{The $\Omega$ and $\omega$-games}\label{sec:omega}

The sequence of number of prime factors counted with multiplicity, is called $\Omega(n)$. Otherwise, when only the distinct primes are counted, it is called $\omega(n)$.\footnote{That is, if~$n$ has canonical form $n=p_1^{a_1}p_2^{a_2}\cdots p_k^{a_k}$, $\Omega(n)=a_1+a_2+\cdots+a_k$ and $\omega(n)=k$.}

Somewhat surprisingly, nim-value sequences for the games that count the number of prime divisors, do not yet appear in OEIS. 

\subsubsection{Move-to the number of prime divisors}\label{sec:mOmega}
The nim-value sequence of {\sc m}$\Omega$ starts at a heap of size one, of nim-value $0$, by definition. Any prime, has a move-to one, so all primes have nim-value one, a square has a move-to the heap of size 2, and hence has nim-value $0$, and so on. The nim-value sequence starts:
 $$0,1,1, 0,1, 0, 1, 0, 0,0,1,0,1,0,0,1,\ldots $$
 
 The indices of the ones is a generalization of the primes: $$2,3,5,7,11,13,16, 17, 19, 23,24, 29, 31, 36, 37, 40, 41 \ldots$$
 
 The number $64$ is in the sequence, and this distinguishes it from A026478. Still it is not exactly A167175, since not all numbers with a nonprime number of prime divisors are included. The sequences coincide until $2^{16} -1$ though, since the first such number to be excluded is $2^{16}$. Via a similar (but easier) reasoning as in Section~\ref{sec:mtau}, we have the following observation.
 
 \begin{observation}
Consider {\sc m}$\Omega$. If $n< 2^{16}$, then $\SG(n)=1$ if and only if $n$ consists of a nonprime number of prime divisors, counted with multiplicity. 
 \end{observation}
 
\subsubsection{Subtract the number of prime divisors}\label{sec:sOmega}
Here we consider the ruleset {\sc s}$\Omega$ `subtract the number of prime divisors'. 
A heap of size one has nim-value $0$, by definition. A heap of size two has a move to a heap of size one, and has nim-value one. A heap of size three has a move to a heap of size two, and has nim-value $0$. The nim-value sequence starts: $$0,1,0,0, 1, 1,0, 0, 1, 1, 0, 0, 1,1,0,1, 0, 1,0,1, \ldots , $$ and the indices of the ones are located at $2,5,6,9,10,13,14,16, 18, 20, 21,23 \ldots $

This sequence does not appear in OEIS.

\subsubsection{Move-to the number of distinct prime divisors}\label{sec:momega}

The nim-value sequence of {\sc m}$\omega$ `move-to number of distinct prime divisors' starts at one, of nim-value zero. The first few nim-values are: $0, 1,1,1, 1, 0, 1, 1, 1, 0, 1, 0,\ldots$, and the corresponding indices of the ones are  $2,3,4,5,7,8,9,11,13,\ldots$

The first nim-value that distinguishes it from {\sc m}$\Omega$ is for the heap of size 4. Since it has only one distinct factor, this game behaves like a prime, and the nim-value is one. Six is the first number that has more than one distinct factor. Hence $7!$ is the smallest number with distinct factors, for which the nim-value is one. 

 \begin{observation}
 If $n< 7!$, then $\SG(n)=1$ if and only if $n$ contains exactly one distinct factor. 
 \end{observation}

\subsubsection{Subtract the number of distinct prime divisors}\label{sec:somega}
The nim-value sequence of {\sc s}$\omega$ `subtract the number of distinct prime divisors' starts at one, which does not have any prime divisor, and hence of nim-value zero. Next, two has the option one, three has the option two, and four has the option three. The first few nim-values are: $$0,1, 0,1, 0,0,1,0,1,1,0,0,1,1,0,1, 0, 0, 1, 1, 0,0,1,\ldots , $$ with $1$s at indices $2,4,7, 9,10,13,14,16,19,20,23,\ldots$ Neither of these sequences appear in OEIS.

\section{Dividing games}\label{sec:dividingG}
The ruleset {\sc dividing} deploys the notion of a disjunctive sum in their recursive definition. That is, an option is typically, with some exception, a disjunctive sum of games. A reference that goes into detail of such games is \cite{DDLP}.

\subsection{The dividing game}\label{sec:dividing}
For this game, the position is a natural number. The $Y$ in the definition of opt is not $2^X$, but instead consists of disjunctive sums of natural numbers. A player divides the current number into equal parts and we write ``+'' to separate the parts, the new game components. To avoid long chains of components, we use multiplicative notation, in the sense that $x\times y$ means $y$ copies of $x$ (that is, $x+\ldots + x$). In this notation, addition is commutative, but multiplication is not. For example $\text{opt}(10)=\{5\times 2, 2\times 5,  1\times 10\}$. The current player moves in precisely one of the components and leaves the other ones unchanged. For example, a move from $5+5$ is to $5+1\times 5=5$ (because no move is possible from $1\times 5$), and indeed, by symmetry, this is the only admissible move.  The number of options is $\tau (n)-1$.

\vskip 8pt
\begin{tabular}{ccc}
$n$ & $\mbox{opt}(n)$ & $\SG(n)$\\\hline
1& $\varnothing$ & 0\\
2& $1+1$ & 1\\
3& $1\times 3$ & 1\\
4&  $2\times 2, 1\times 4$  & 1\\
5& $1\times 5$ & 1\\
6& $3\times 2, 2\times 3, 1\times 6$ & 2\\
7& $ 1\times 7$ & 1\\
8& $ 4\times 2, 2\times 4, 1\times 8$ & 1
\end{tabular}\vskip 8pt\noindent

Let $\Omega_2(n)$ denote the number of prime factors of $n$, where the powers of 2 are counted without multiplicity, and the powers of odd primes are counted with multiplicity.
\begin{theorem}
Consider {\sc dividing}. For all $n\in \N$, $\SG(n)=\Omega_2(n)$.
\end{theorem}
\begin{proof}
$\SG(1)=0$, and $1$ does not have any prime components. Suppose that $n$ is a power of two. Then $\SG(n)=1$, since each option permits the mimic strategy. Similarly, if $n$ is a prime, then $\SG(n)=1$. Suppose that $n=2^kp_1\cdots p_j$, with $k\in \N_0$ and each $p_i$ odd. We use induction to prove that $\SG(n)=j$ if $k=0$, and otherwise $\SG(n)=j+1$. If $k=0$, $n$ can be split into an odd number of components each having  $m$ prime factors for each $m\in [1, j-1]$. Induction and the nim-sum together with the mex-rule gives the result in this case. Similarly, if $k>0$, $n$ can be split into an odd number of components of $m$ prime factors for each $m\in [1, j]$, which proves that $\SG(n)=j+1$ in this case. 
\end{proof}

\begin{example}
Suppose the position is $18+7=2\cdot 3^2+7$. How do you play to win? The nim-value is $3\oplus1=2$, where $\oplus$ denotes the nim-sum. Hence the next player has a good move. The good move turns the 18-component to nim-value 1, that is, we divide it into an odd number of even numbers with no odd factor. This can be done in only one way: you move to $2\times 9 + 7=2+7$, and clearly $\SG(2+7)=0$. The next player has exactly two options, but, either way, you will finish the game in your next move.
\end{example}

\subsection{The dividing and remainder game}\label{sec:divrem}
The ruleset {\sc divide-and-residue} is an extension of {\sc dividing}, where you are allowed to divide $n$ to $k$ equal parts $d$ and a remainder $r$ that is smaller than the parts. Thus, here we have a lot more options (for a generic game) than {\sc dividing}, which is obvious by the representation $n=k\times d+r$, with $0\le r< d$. By moving we are free to choose any $1\le d<n$, so we have $n-1$ options, for all $n>0$. The \SG-sequence starts: 

\vskip 8pt
\begin{tabular}{ccc}
$n$ & $\mbox{opt}(n)$ & $\SG(n)$\\\hline
1& $\varnothing$ & 0\\
2& $1+1$ & 1\\
3& $2+1,1\times 3$ & 2\\
4&  $3+1, 2\times 2, 1\times 4$  & 1\\
5& $4+1,3+2,2\times 2+1,1\times 5$ & 2\\
6& $5+1, 4+2, 3\times 2, 2\times 3, 1\times 6$ & 3\\
7& $6+1, 5+2, 4+3, 3\times 2+1, 2\times 3+1, 1\times 7$ & 2\\
8& $7+1,6+2,5+3, 4\times 2, 3\times 2+2, 2\times 4, 1\times 8$ & 3
\end{tabular}\vskip 8pt\noindent

An even number of heaps of the same sizes reduces to a heap of size one. A heap of size one in a disjunctive sum, gets removed. We get an  equivalent reduced table: \vskip 8pt\noindent

\begin{tabular}{ccc}
$n$ & $\mbox{opt}(n)$ & $\SG(n)$\\\hline
1& $\varnothing$ & 0\\
2& $1$ & 1\\
3& $2, 1$ & 2\\
4&  $3,1$  & 1\\
5& $4,3+2,1$ &2\\
6& $5, 4+2, 2, 1$ & 3\\
7& $6, 5+2, 4+3, 2, 1$ & 2\\
8& $7,6+2,5+3, 2, 1$ & 3\\
9& $8,7+2,6+3, 5+4, 3,1$&4\\
10& $9,8+2,7+3,6+4,3,2,1$&3
\end{tabular}\vskip 8pt

Note that, when we remove pairs of equal numbers, sometimes we must add the option `1' to symbolize a move to a terminal position of nim-value $2+2=0$. From this table we may deduce that the game $7+3$ from the first paragraph in the paper is indeed a losing position. 
{\sc divide-and-residue} has a mysterious \SG-sequence, as depicted in Figure~\ref{fig:divares}.

\begin{figure}[h!]
 \centering{
    \includegraphics[width=1\textwidth]{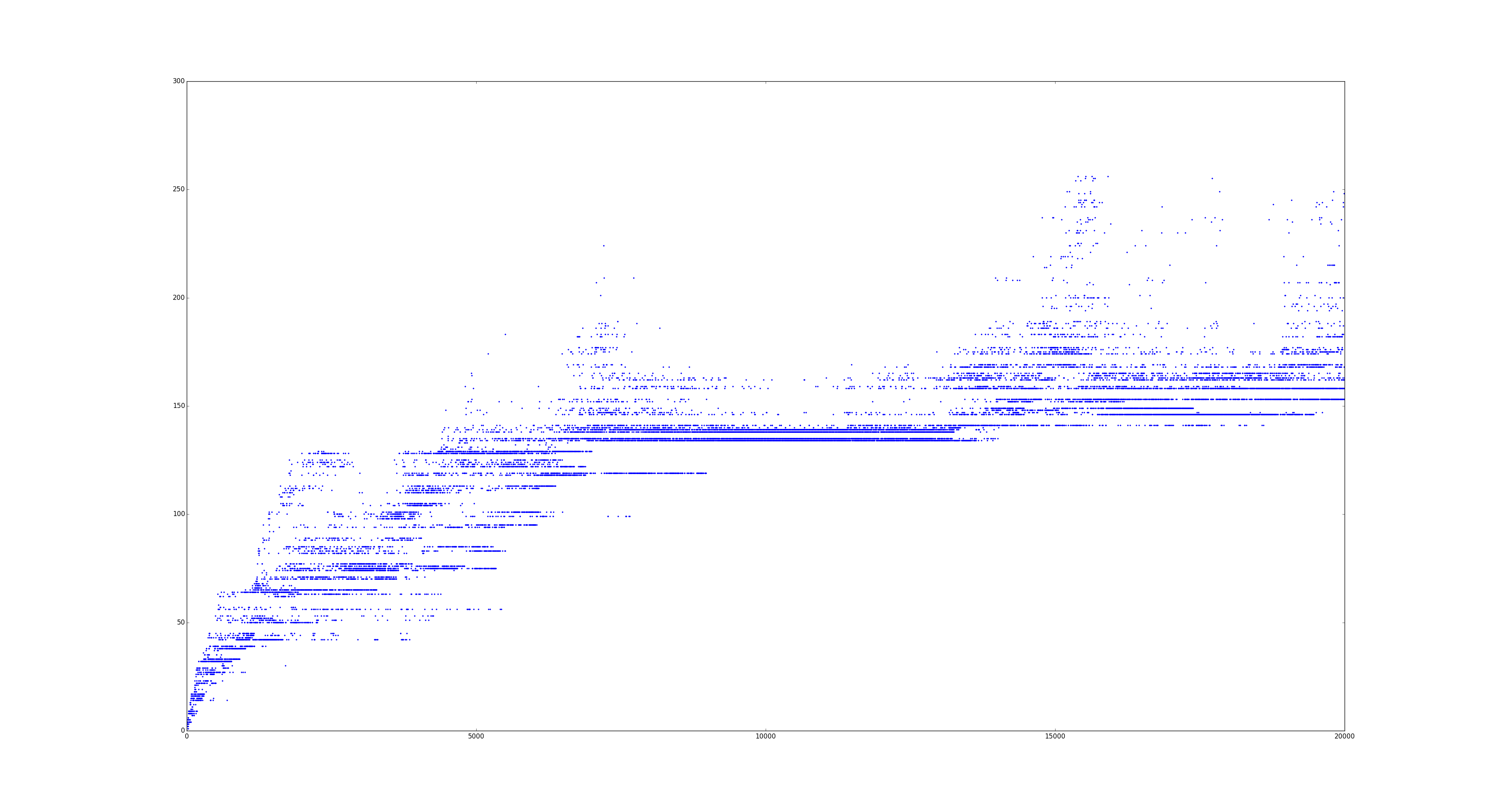}}
    \caption{The initial 20000 nim-values of {\sc divide-and-residue}. They just about touch the nim-value $2^8=256$.}\label{fig:divares}
\end{figure}

Here are the 50 first nim-values, of the form [heap size, nim-value]:

$[1,0],
[2, 1],
[3, 2],
[4, 1],
[5, 2],
[6, 3],
[7, 2],
[8, 3],
[9, 4],
[10, 3],
[11, 4],
[12, 3],
[13, 4],
[14, 3],
[15, 4],
[16, 3],\newline
[17, 4],
[18, 5],
[19, 4],
[20, 5],
[21, 3],
[22, 5],
[23, 4],
[24, 2],
[25, 1],
[26, 5],
[27, 6],
[28, 5],
[29, 6],
[30, 2],
[31, 6],\newline
[32, 5],
[33, 3],
[34, 8],
[35, 9],
[36, 8],
[37, 9],
[38, 8],
[39, 9],
[40, 8],
[41, 9],
[42, 4],
[43, 9],
[44, 4],
[45, 9],
[46, 8],\newline
[47, 9],
[48, 4],
[49, 9],
[50, 4].$

Early nim-values tend to be odd for heaps of even size, and even for those of odd size. By an elementary argument we get: the heap of size 25 is the largest heap of nim-value one, and one can prove the analogous statement for a few more small nim-values. We conjecture that any fixed nim-value occurs finitely many times. 

For the upper bound, the nim-values seem to be bounded by $n^{3/5}$, for sufficiently large heap sizes $n$. An empirical observation is that the growth of nim-values appears to be halted at powers of two. For example, the nim-value $2^2$ starts to appear at heap size $9$, but does not increase beyond $2^2+2^0$ until the heap of size 27. 

\begin{conjecture}\label{con:1}
Consider {\sc divide-and-residue}. Then each nim-value occurs, and at most a finite number of times. Moreover $\SG(n)/n\rightarrow 0$, as $n\rightarrow \infty$. 
\end{conjecture}

There is no big surprise that this game is hard, since it is an extension of {\sc grundy's game} \cite{WW, Sloane}. 
Indeed, the options of {\sc divide-and-residue} in which the divisor $d$ is greater than $n/2$ correspond to the rule of splitting a heap into two unequal parts of {\sc grundy's game}. If we define the ruleset {\sc complement-grundy}, by requiring that $k\ge 2$ in {\sc divide-and-residue}, then we can prove the second statement in Conjecture~\ref{con:1} for this new game. 
Let us tabular the first few nim-values, where options are displayed in reduced form:

\begin{tabular}{ccc}
$n$ & $\mbox{opt}(n)$ & $\SG(n)$\\\hline
1& $\varnothing$ & 0\\
2& $1$ & 1\\
3& $1$ & 1\\
4&  $1$  & 1\\
5& $1$ &1\\
6& $2, 1$ & 2\\
7& $2, 1$ & 2\\
8& $2, 1$ & 2\\
9& $3, 2,1$&2\\
10& $3,2,1$&2\\
11& $3,3+2,2,1$&2\\
\end{tabular}\vskip 8pt

Figure~\ref{fig:compgru} shows that the initial regularity of nim-values is replaced by more complexity further down the road, although not as severely as for {\sc divide-and-residue}. Note that the two games appear to share some geometric properties such as a local stop of nim-value growth at powers of two, and a bounded number of occurrences for each nim-value. In this case though, some nim-values do not appear, such as $12, 15, 20$ etc. We do not yet know if  the omitted nim-values can be described by some succinct formula, and we do not even know if the occurrence of each nim-value is finite.
\begin{figure}[ht!]
 \centering{
    \includegraphics[width=1\textwidth]{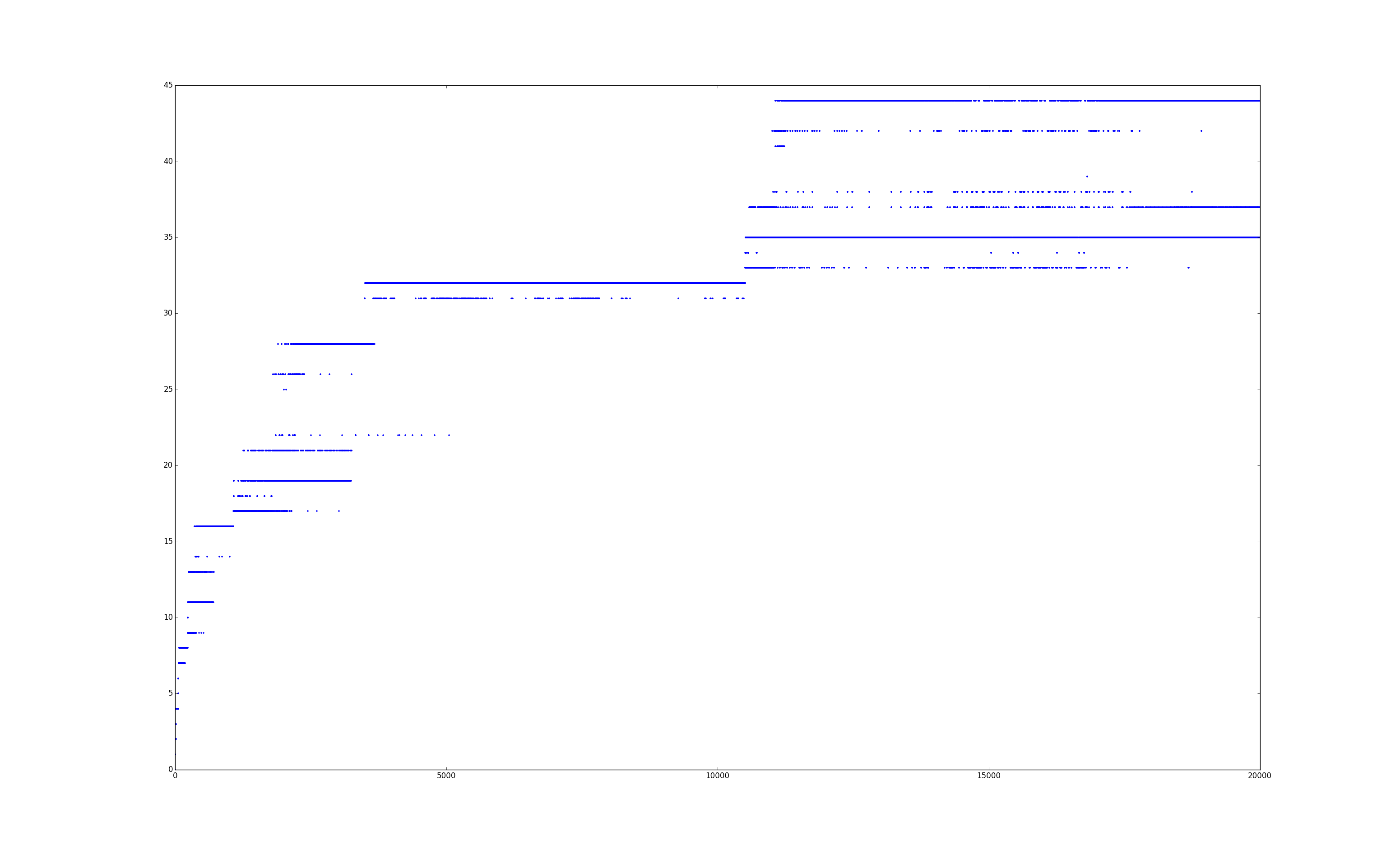}}
    \caption{The initial $20000$ nim-values of {\sc complement-grundy}.}\label{fig:compgru}
\end{figure}

\begin{theorem}
Consider {\sc complement-grundy}.  Then $\SG(n)/n\rightarrow 0$, as $n\rightarrow \infty$. 
\end{theorem}

\begin{proof}
Consider the nim-value $2^k$. If it does not appear, we are done. Suppose it appears for the first time at heap size $n_k$. By the mex rule, if a nim-value is greater than $2^k$, it must have nim-value $2^k$ in its set of options. By the rules of game, this can only happen for a heap of size $m\ge 3\cdot n_k$. In particular, this holds for the nim-value $2^{k+1}$, which occurs for the first time at $m=n_{k+1}$, say. Thus, for nim-values that are powers of two, we get $\frac{2}{3}\SG(n_k)/n_k\ge \SG(n_{k+1})/n_{k+1}$. This upper bound holds for arbitrary nim-values, since the lower bound on where the power of two $2^{k+1}$ can appear is the same lower bound where any other nim-value greater than $2^k$ may appear.
\end{proof}

Two simpler variations of {\sc  divide-and-residue} are:\\

\noindent 1) The remainder is not included in the disjunctive sum of an option: {\sc divide-throw-residue}.\\ 

\noindent 2) Only the remainders are the options: {\sc residue-throw-divisor}.\\

\begin{figure}[hb!]
  \centering
\includegraphics[width=.50\textwidth]{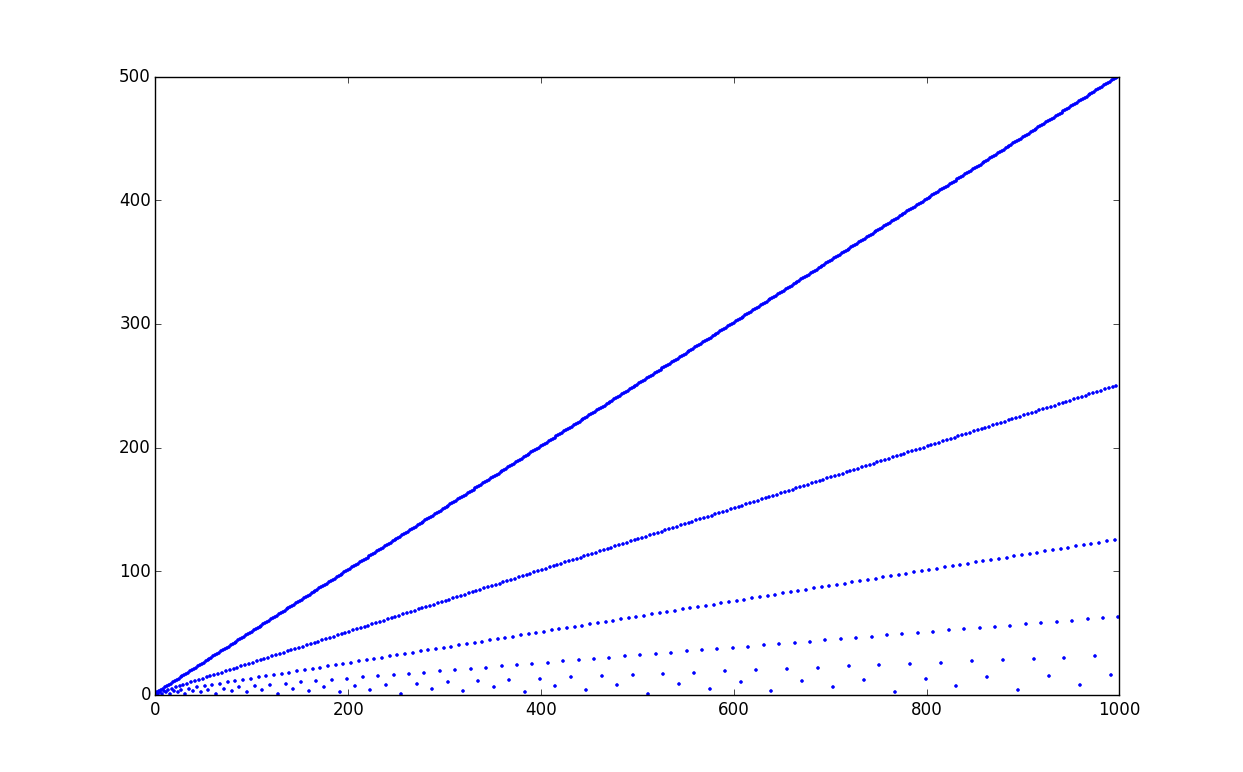}
    \includegraphics[width=.42\textwidth]{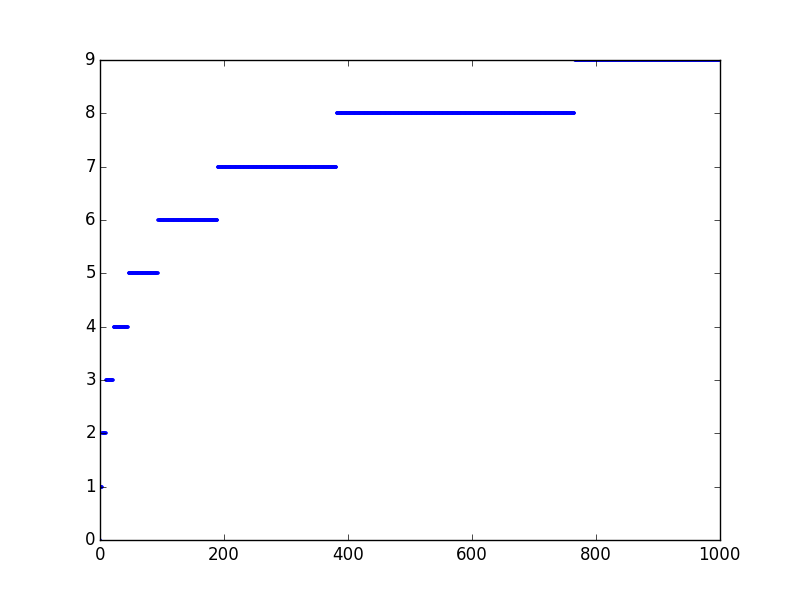}
\caption{The initial nim-values of {\sc divide-throw-residue} and {\sc residue-throw-divisor}, respectively.}\label{fig:divnoresresnodiv}
\end{figure}

The patterns of the nim-values of these rulesets are displayed in Figure~\ref{fig:divnoresresnodiv}. For variation~1 (to the left), we prove that, for heaps $>1$, the \SG-sequence coincides with OEIS, A003602: If $n = 2^m(2k-1)$, for some $m\ge 0$, then $a(n) = k$. The Sprague-Grundy sequence starts at heap of size one with nim-values as follows $$0,1,2,1,3,2,4,1,5,3,6,2,7,4,8,1,9,5,10,3,11,6,12,\ldots$$ 

It turns out the {\sc divide-throw-residue} has the same solution as {\sc maliquant}, the game where the options are the non-divisor singletons; recall  Theorem~\ref{thm:maliquant}, where this result is expressed as an index function, $i_o$, the index of the largest odd divisor.

\begin{theorem}
Consider {\sc divide-throw-residue}. Then $\SG(n)=i_o(n)=k$, if $n= 2^m(2k-1)$, for some integer $m\ge 0$.
\end{theorem}
\begin{proof}
Observe that the options in the interval $[\lfloor n/2\rfloor+1, n-1]$ are the same as for {\sc maliquant}. Assume first $n$ is even. Then $n/2+n/2$ is an option in {\sc divide-throw-residue}, but $n/2$ is not an option in {\sc maliquant}. However, $n/2+n/2$ only contributes the nim-value $0$ and may be ignored. Consider next the disjunctive sum $m+\cdots +m$, with an odd number of components adding up to $n$. Then there is a power of 2, say $2^k$  such that $2^km\in [\lfloor n/2\rfloor+1, n-1]$, i.e., $2^km\nmid n$. And so, by induction, $\SG(m) = \SG^{\rm M}(2^km)$, where the M indicates {\sc maliquant}. On the other hand, there are options of {\sc maliquant} of the form $m\nmid n,  m<n/2$. They do not have a match in  {\sc divide-throw-residue}. But, as we saw in the proof of Theorem~\ref{thm:maliquant}, they do not contribute to the nim-value computation in {\sc maliquant}. The case of odd $n$ is similar.
\end{proof}

For variation~2, we observe the following \SG-sequence: $0,1,1,1,2,2,2,2,2,2,\ldots$, i.e., for $n>0$ if $3\cdot 2^k$ copies of $k+1$ have appeared append $3\cdot 2^{k+1}$ copies of $k+2$ as the next nim-values. 

\begin{theorem}
Consider {\sc residue-throw-divisor}. For all $n\in \N$, $\SG(n)=k$ if $$n\in \{3(2^{k-1}-1)+2,\ldots , 3(2^k-1)+1 \}.$$
\end{theorem}

\begin{proof}
We leave this proof to the reader.
\end{proof}

\section{The factoring games}\label{sec:factoring}
Is there any game that has the number of prime factors of $n\in \mathbb{N}$ as the sequence of nim-values? The answer is yes, as given by the almost trivial aliquot game in Section~\ref{sec:maliquot}. There is a related game that decomposes into several components in play, namely to play to any factorization of $n$.

\begin{example}[{\sc m-factoring}]\label{move-to factoring}
Let $n=12$. Then the set of options is $\{6 + 2, 3 + 4, 2+ 2+ 3\}$. The unique winning move is to $2+ 2+ 3$, because the nim-values in the set of options are $1,1,0$, respectively. Hence $\SG(12) = 2$.
\end{example}

\begin{example}[{\sc s-factoring}]\label{subtract factoring}
Let $n=12$. Then the set of options is $\{6 + 10, 9 + 8, 10+ 10+ 9\}$. The \SG-sequence starts:  $0,0,1,1,1,1,1,1,2,1,1,1,2,1,1,1,1,1,2,1,2,1,1,1,1,1$, where the first heap is the empty heap and the second 0 is due to that 1 does not have any prime factors. $\SG(12)=2$.  
\end{example}

{\sc m-factoring} has a simple solution, but {\sc s-factoring} we do not yet understand. Recall the omega-functions from Section~\ref{sec:omega}. 
\begin{theorem} Consider {\sc m-factoring}, and let $n\geqslant 2$, where each option is a non-trivial disjunctive sum of a factoring of $n$. Then $\SG(n)=\Omega(n)-1$. If no two distinct components may contain the same prime number, then $\SG(n)=\omega(n)-1$.
\end{theorem}

\begin{proof}
If $n$ is a prime, then $\SG(n)=0$, because no factoring to smaller components is possible. If $n$ is composite with $k$ prime factors, then, by induction, it is possible to play to an option of nim-value $\ell$, for each $\ell\in \{1,\ldots , k-2\}$, by factoring $n$ into one number with $\ell$ prime factors, and $k-\ell$ other prime components. On the other hand, the nim-value $k-1$ cannot be obtained as a move option, since $(x_1-1)\oplus \cdots \oplus (x_\ell-1)\leqslant  (x_1-1) + \cdots + (x_\ell-1)\leqslant k-2$, if $k=x_1+\cdots +x_\ell$. The proof of the second part is similar.
\end{proof}

\section{Full set games}\label{sec:fullset}
In {\sc fullset maliquot} a player moves to all the proper divisors in a disjunctive sum.\footnote{Obviously we need to exclude the divisor $n\mid n$; the word ``proper" is implicit in the naming.}  Let us display the first few numbers with their options and nim-values.

\vskip 8pt
\begin{tabular}{ccc}
$n$ & $\mbox{opt}(n)$ & $\SG(n)$\\\hline
1& $\varnothing$ & 0\\
2& $1$ & $1$\\
3& $1$ & $1$\\
4& $1+2$ & $0$\\
5& $1$ & $1$\\
6& $1+2+3$ & $1$\\
7& $1$ & $1$\\
8& $1+2+4$ & $0$\\
9& $1+3$ & $0$
\end{tabular}\vskip 8pt\noindent

The nim-value sequence starts $0, 1, 1, 0, 1, 1, 1, 0, 0, 1, 1, 0, 1, 1, 1, 0, 1, 0, 1, 0, 1, 0, 1, 0, \ldots $. The non-unit proper divisors of $24$ are $2,3,4,6,8$ and $12$. The only square-free ones are $2,3$ and $6$, an odd number. Such observations are relevant for the proof of the location of the 0s.

\begin{theorem}
Consider {\sc fullset maliquot}. Then $\SG(n)\in\{0,1\}$, and $\SG(n)=1$ if and only if  $n>1$ is square-free.
\end{theorem}
Let us indicate the idea of the proof. The nim-value $\SG(4)=0$ because the only non-unit proper divisor, $2$, is square-free, and $\SG(8)=0$, because there is exactly one square-free proper divisor, namely 2. In the proof we will use the idea that $\SG(n) = 0$ if and only if $n$ has
an even number of square-free proper divisors. 
\begin{proof}
We induct on the number of divisors. 

If $n=p$ is prime, there is an even number, namely 0, of square-free non-unit proper divisors. The nim-value $\SG(p)=1$ is correct, because the move to the heap of size one is terminal. 

Consider an arbitrary number $n$. Each move will alter the nim-value modulo 2. We must relate this to the non-unit square-free proper divisors in the components of the option of $n$. By induction, if this number is even if and only if $\SG(n)=0$, we are done. Henceforth, we will ignore the component of a heap of size one, since it has nim-value $0$ and will not contribute to the disjunctive sum.

Suppose first that $n=p^2$ is a perfect square. Then the option is the prime $p$, and hence $\SG(p^2)=0$. Indeed, there is an odd number of square-free non-unit divisors.

If $n=p^t$, $t>2$, is any other power of a prime, we must prove that $\SG(n)=0$. The set of non-unit proper divisors is $\{p,\ldots , p^{t-1}\}$, and hence there is exactly one square-free divisor in the disjunctive sum  $p+\cdots + p^{t-1}$. By induction, we get that each component, except $p$ has nim-value $0$. This proves this claim.

Next, suppose $n = pq$, where $p$ and $q$ are primes. Then the option is $p + q$ of nim-value $1\oplus 1 = 0$. Hence $\SG(pq)=1$, and $n$ has an even number of square-free non-unit proper divisors.

Similarly, if $n=p_1\cdots p_j$ is a product of distinct primes, then $\SG(n)=1$. This follows, because the number of proper non-unit divisors, 
\begin{align}\label{eq:numberdivisors}
\sum_{1\le i<j} {j\choose i},
\end{align}
 is even, where $j$ is the number of prime factors in $n$ (this holds both for even and odd $n$). And, by induction, each such individual component divisor has nim-value 1. Note that, by moving in one such divisor, the number of components in the disjunctive sum changes parity; if moved in a prime, then the prime is deleted, if moved in $pq$, then this component splits to $p+q$, and so on. 

By combining these observations, we prove the general case of an arbitrary prime factorization. Assume $n$ contains a square. We must show that $\SG(n)=0$. By induction, we are concerned only with the square-free divisor components, and we show that the number of such divisors is odd. 

Indeed, if we assume $j$ in \eqref{eq:numberdivisors} is the number of distinct prime factors, then there is one missing term, namely ${j \choose j}$. Namely, the divisor composed of all square-free factors must be counted, whenever $n$ contains a square. Apart from this, no new square-free divisor is introduced. Thus, the number of such components is odd, and since by induction they have nim-value $1$, the result $\SG(n)=0$ holds.
\end{proof}

We have investigated a few more of the fullset games, including those in the subclass `subtraction', but not yet found other examples with sufficient regularity to prove basic correspondence with number theory. Apart from {\sc fullset maliquot}, this class, for now, remains a mystery.
For example, for {\sc fullset totient}, the sequence starts
 $0, 1, 0, 1, 1, 0, 0, 0, 0, 1, 1, 1, 1, 0, 1, 1, 0, 0, 0$. The heap of size one has nim-value zero by definition, and the heap of size two has nim-value one, because one is relatively prime with two. $\SG(3)=0$, because the option is $1+2$ of nim-value $0\oplus 1=1$. The sequence of the indices of the ones is $2,4,5, 10, 11,12, 13, 15$, and so on. This sequence does not yet appear in OEIS.
 
\section{Powerset games}\label{sec:powerset}
We study six version of the powerset games on arithmetic functions, and  we begin by listing the first 20 nim-values for the respective ruleset. All start at a heap of size one, except item 2, which starts at the empty heap (defined as terminal).
\begin{enumerate}
\item {\sc powerset maliquot}: move-to an element in the powerset of the proper divisors. $$0, 1, 1, 2, 1, 2, 1, 4, 2, 2, 1, 4, 1, 2, 2, 8, 1, 4, 1.$$
\item {\sc powerset saliquot}: subtract an element in the powerset of the divisors. $$ 0, 1, 2, 1, 4, 1, 2, 1, 8, 1, 2, 1, 4, 1, 2, 1, 16, 1, 2, 1.$$
\item {\sc powerset maliquant}: move-to an element in the powerset of the non-divisors. $$0, 0, 1, 0, 2, 1, 4, 8, 16, 2, 32, 1, 64, 4, 128, 8, 256, 16, 512$$
\item {\sc powerset saliquant}: subtract an element in the powerset of the non-divisors. $$0, 0, 1, 1, 2, 1, 4, 4, 8, 2, 16, 8, 32, 32, 64, 64, 128, 8, 256, 64.$$
\item {\sc powerset totative}:  move-to an element in the powerset of the relatively prime numbers smaller than the heap. $$0, 1, 2, 1, 4, 1, 8, 1, 2, 1, 16, 1, 32, 1, 2, 1, 64, 1, 128.$$
\item {\sc powerset nontotative}: move-to an element in the powerset of the non-relatively prime numbers smaller than the heap. $$0, 0, 0, 1, 0, 2, 0, 4, 1, 8, 0, 16, 0, 32, 4, 64, 0, 128, 0.$$
\end{enumerate}

For single heaps, these games tend to have nim-values powers of two. The intuition of this is that by induction there is plenty opportunity, in a powerset, to construct any number between the powers of two, by using various sums of single heaps. We will study the precise behavior in a couple of instances, namely items 2,3 and 5.

\begin{theorem}
Consider {\sc powerset saliquot}. Then $\SG(0)=0$, and $\SG(n)=2^p$, if $2^p$ is largest power of two divisor of $n\ge 1$.
\end{theorem}
\begin{proof}
A heap of size zero has nim-value 0 because it is terminal by definition. The heap of size one has nim-value $1=2^0$, because $1\mid 1$.  The heap of size two has nim-value $2=2^1$, because $1, 2\mid 2$, and $\SG(2-1)=1, \SG(2-2)=0$. Both these cases satisfy the largest power of two divisor criterion.

Suppose the statement holds for all numbers smaller than the heap size $n=2^pa$, with $a$ odd, say. We must show that all nim-values less than $2^p$ exist among the options of $n$. For each $q<p$, we will find a number $0 \le m< n$ with $2^q$ largest power of $2$ divisor of $m$, and where $n-m$ is a divisor of $n$. For example with $m=n-2^qa\in \N$, then $n-m=2^qa\mid n$, and $m=2^{q}a(2^{p-q}-1)$ has greatest power of two divisor $2^q$. By induction, $\SG(m)=2^q$. Let $q$ range between $0$ and $p-1$. By the rules of {\sc powerset}, and by using the disjunctive sum operator, this suffices to establish that all nim-values less than $2^p$ exist among the options of $n$.

Next, we must prove that the nim-value $2^p$ does not exist among the options. It suffices to show that no individual heap in an option, which is a disjunctive sum, is of the same form as $n$. This follows, since, by induction, all numbers smaller than $n$ have nim-values powers of two, and apply nim-sum. 

A divisor of $n$ is of the form $2^qy$, where $y\mid a$ is odd, and where $q\le p$. 

Suppose first $q=p$. Then $n-2^py=2^py(a/y-1)$. But $a/y$ is odd, and hence $a/y - 1$ is even, so $n - 2^py = 2^zb$, with $z>p$ and $b$ odd, unless $a=y$ when $n-2^py=0$. 

In case $q < p$, we get $n-2^qy=2^qy(2^{-q}n/y-1)$, and since $2^{-q}n/y-1$ is odd, by induction, the heap is not of the same form (since $q<p$). 
\end{proof}
Recall the indexing function, $i_o$, of largest odd divisor, concerning the singleton version of {\sc maliquant}. It applies here as well with some initial  modification; while it looks like one could `peel' off the $2$'s it does not work due to the irregular set of initial nim-values.
\begin{theorem}
Consider {\sc powerset maliquant}. The sequence starts at a heap of size one, and the first eight nim-values are, $0,0,1,0,2,1,4,8$. Otherwise, if $n=2k+1, k\ge 4$, then $\SG(n)=2^k$, and if $n\ge 10$ is even, then $\SG(n)=\SG(n/2)$. 

\end{theorem}
\begin{proof}
The smaller heaps are easy to justify by hand. 
The heap of size 8 is pivotal. It achieves nim-value 0, by the option $3+6$, both numbers being nondivisors. And the nim-values $1,2,4$ may be combined freely by using the nondivisor heaps $5,6,7$. Hence, the nim-values of the small heaps are verified. 

For the base cases, we consider the heaps of sizes $9$ and $10$, of nim-values $16=2^4$, with $2\cdot 4+1=9$ and $2=\SG(5)$ respectively.  

For the induction, let us start with a heap of even size, $n=4t+2$, say. It suffices to show that $\SG(n)=\SG(n/2)$. Observe that each number between $n/2$ and $n$ is a nondivisor to $n$, and hence may be part of a disjunctive sum to build desirable nim-values, by induction. Since $n/2 = 2t+1$ is odd, each power of two $2^4, \ldots , 2^t$ appears among the nim-values for heap sizes in $[9, n/2]$. By induction, each power of two $2^0,\ldots , 2^{t-1}$ appears as a nim-value in the heap interval $I=[n/2-1,\ldots n-1]$. Namely, for $y\in [0,t-1]$, multiply $2y+1$ by $2$ iteratively until $2^s(2y+1)\in I$.  Thus all numbers smaller than $2^{t}$ appears as options, but note that the nim-value $2^t$ appears only as a nim-value for a divisor of $n$, and hence this is the minimal exclusive. This proves that $\SG(n)=2^t$, if $n$ is even, as desired.  

Now, consider odd $n=2t+1$, say. By induction the nim-value of each heap smaller than $n$ is less than $2^t$. In case of $t$ even, the powers of two, $2^{t/2},\ldots , 2^{t-1}$ appear for nim-values of odd heaps in the interval $[t,\ldots , 2t-1]$. And similar to the case for even $n$, the smaller power of two nim-values can also be found in this interval. Therefore each nim-value smaller than $2^t$ appear as an option of a disjunctive sum of non-divisors of $n$. Hence, the minimal excusive is $\SG(n)=2^t$.
\end{proof}
Recall the function $i_p$, the index of the smallest prime divisor of $n$, where the prime 2 has index 1,  for the solution of {\sc totative}, from Section~\ref{sec:totative}. It applies for the powerset game as well.
\begin{theorem}\label{thm:powtotative}
Consider {\sc powerset totative}. Then $\SG(n)=2^{i-1}$, where $i=i_p$.
\end{theorem}
\begin{proof}
The nim-value of a heap of size one is 0, since it is terminal. A heap of size two has a move to the heap of size one, because 1 is relatively prime with all numbers greater than 1. Hence $\SG(2)=1=2^0$. Suppose the statement holds for all numbers smaller than $n>1$. 

If $n$ is even, we must prove that there is a move to nim-value 0, but no move to nim-value 1. The first part is done in the first paragraph. Hence, let us show, by induction, that there is no move to nim-value 1. Since all smaller heaps of odd size have even nim-values, then a disjunctive sum of nim-value 1 must contain a heap of even size. This is impossible, since heaps of even size are not relatively prime with $n$. 

Suppose that $n$ is odd, so that the index of the smallest prime divisor of $n$ is $i>1$. We must show that $\SG(n)=2^{i-1}$. By induction, each smaller prime divisor, with index $q<i$ say, has appeared in a heap size smaller than $n$,  with nim-value $2^{q-1}$. Since any disjunctive sum of heap sizes relatively prime with $n$ is permitted as an option, by induction, each nim-value smaller than $2^{i-1}$ can be obtained. 

Next, we show that there is no option of nim-value $2^{i-1}$. This generalizes the idea used in the second paragraph. A disjunctive sum of nim-value $2^{i-1}$ must contain a component of nim-value $2^{i-1}$. But, by induction, those heap sizes are not relatively prime with $n$.
\end{proof}

\section{Discussion--future work}\label{sec:disc}
A natural generalization of counting the number of elements satisfying an arithmetic function is to instead consider their sum, or partial sums. For example, consider the sum generalization of the {\sc mtau}, that is, the option of $n$ is the sum of the proper divisors of $n$. 
For example $4$ has the proper divisors $1$ and $2$ and therefore the option is $3$. Loops and cycles occur for perfect numbers (those where the sum of proper divisors equals the number) and (temporarily) increased heap sizes for abundant numbers
(those where the sum of proper divisors is greater than the number).  The first loop appears at $1+2+3 = 6$ (where `+' is arithmetic sum). 

This might at first sight seem to disqualify the Sprague-Grundy function,\footnote{Fraenkel et al have developed a generalized Sprague-Grundy function for cyclic short games.} but in fact, since the game is binary, the cycles are trivial, in the following sense. If we play a disjunctive sum of games where one component will not end, then the full game will not end. And reversely, if no component contains a cycle, but perhaps temporarily increasing heap sizes, then the full game will end and a winner may be declared. The nim-value sequence of this ruleset begins at a heap of size one, as follows: $0, 1, 1, 0,1,\infty, 1, 0, 1, 1, 1, ?$, where the infinity at heap size 6 indicates the loop, $1+2+3=6$. 

Let us compute the nim-value for $n = 12$, which is indicated by `?' in the sequence above. The sum of proper divisors is $16$ 
(temporary increase), followed by options $15$ and $9$, in the next two moves. The sequence above indicates that $\SG(9) = 1$, and therefore, $\SG(12) = 0$. The recurrence where a number is mapped to the sum of its proper divisors has been studied in number theory literature, without the games' twist. It seems well worthy some more attention.

Even more interesting is the same ruleset but where the player may pick any partial sum of proper divisors. We have the following table, where for example the options of a heap of size $4$ are $1,2$ and $1+2$.
\vskip 8pt
\begin{tabular}{ccc}
$n$ & $\mbox{opt}(n)$ & $\SG(n)$\\\hline
1& $\varnothing$ & 0\\
2& $1$ & $1$\\
3& $1$ & $1$\\
4& $1,2,3$ & $2$\\
5& $1$ & $1$\\
6& $1,2,3, 4,5,6 $ & $\infty_3$\\
7& $1$ & $1$\\
8& $1,2,3,4,5,6,7$ & $\infty_3$\\
9& $1,3,4$ & $3$
\end{tabular}\vskip 8pt\noindent

Here $\infty_3$, means the nim-value 3, but with an additional option an infinity, namely $\infty_3$. Consider for example the disjunctive sum of heaps $6+9$. Then every move apart from playing to $\infty_3$ is losing. So, this game is a draw. However, playing instead $6+7$, the first player wins by moving to $2+7$, $3+7$ or $5+7$. That is, a loopy game component is sensitive to the disjunctive sum. The $\omega$ game also seems to have an interesting sum variation, but now we are ready to go and prepare some lunch.\\

\noindent{\bf Acknowledgement.} This work started when the second author visited the first author at the University of the Virgin Islands in April 2015. It breaks my heart to acknowledge that the first author passed away 15 October 2020. Doug is deeply missed. Many thanks to the referee, whose comments helped to improve the readability of this paper.

\end{document}